\documentclass{article}
\usepackage[OT2,OT1]{fontenc}
\usepackage{amsfonts,amsmath, amssymb,latexsym,mathrsfs,comment,url}
\usepackage[all,cmtip]{xy}
\usepackage{mathtools}

 \usepackage{graphicx}
 \usepackage{titling}
\usepackage{tikz-cd}
\usepackage{upgreek}
\usepackage{tikz}
\usepackage{amssymb}
\usepackage{comment}
\usetikzlibrary{matrix, arrows}
\usepackage{float}
\def\binom#1#2{{#1\choose#2}}

\def\Z{{\mathbb Z}}
 \DeclareFontFamily{U}{wncy}{}

\def\GL{{\rm GL}}

\def\PGL{{\rm PGL}}

\def\Gal{{\rm Gal}}
\def\Cl{{\rm Cl}}

\def\O{{\mathcal O}}
\def\P{{\mathbb P}}

\def\Aut{{\rm Aut}}

\def\R{{\mathbb R}}
\def\F{{\mathbb F}}

\def\Hom{{\rm Hom}}
\def\Q{{\mathbb Q}}

\def\Z{{\mathbb Z}}
\def\P{{\mathbb P}}
\def\F{{\mathbb F}}
\def\Q{{\mathbb Q}}

\def\CC{{\mathcal C}}

\def\Aut{{\text{Aut}}}

\def\Spec{{\text{Spec}}}
\def\Aut{{\operatorname{ Aut}}}
\def \gcd{{\rm gcd}}

\usepackage{enumitem}

\usepackage{tikz}

\newcommand*{\G}{\mathbb{G}}

\newcommand*{\ra}{\rightarrow}

\newcommand*{\ol}{\overline}

\usepackage{amsthm}

\newtheorem{theorem}{Theorem}[section]

\newtheorem{lemma}[theorem]{Lemma}

\usepackage{amsmath}
\usepackage{tikz-cd}
\usepackage{calc} 
\usetikzlibrary{cd, arrows, matrix, positioning}
\usetikzlibrary{cd}
\usepackage{xcolor}        
\usepackage{marginnote}    

\usepackage{xcolor}
\usepackage{geometry}      
\reversemarginpar          

\newtheorem{proposition}[theorem]{Proposition}
 
\def\pp{{\mathfrak{p}}}

\def\deg{{\text{deg}}}

\newtheorem{definition}{Definition}
\newtheorem{remark}[theorem]{Remark}
\newtheorem{prop}[theorem]{Proposition}

 \usepackage{romannum}
\usepackage{amscd,amssymb,amsmath}
\usepackage{enumitem}
\usepackage{relsize}
\usepackage{color}
\usepackage{amsmath}

\author{Fateme Sajadi}

\title{A Unified Finiteness Theorem For Curves Over Function Fields }
\date{}
\begin{document}
\maketitle
\begin{abstract}
Motivated by the analogy between number fields and function fields, this paper extends the main result of \cite{janbazi2025unified} to the function field setting. Let $C$ be a smooth affine curve over a finite field, and let $\pi: S \rightarrow C$ be a smooth, proper model of a curve over $C$. Then, for any fixed integer $n \in \mathbb{N}$, there are only finitely many horizontal divisors of degree $n$ that are étale over the base $C$, up to the action of the automorphism group and Frobenius (in the isotrivial case).
\end{abstract}

\pagenumbering{arabic}

\section{Introduction}\label{intro}

In \cite{janbazi2025unified}, we presented a result on integral points on curves which unified several fundamental theorems in arithmetic geometry, notably, Faltings' Theorem, Birch-Merriman Theorem, and Siegel's Theorem. In this  work, we seek to generalize and build upon this framework in the context of function fields.
One well-known complication that arises in the function field setting, and which does not occur in the number field case, is the issue of (iso)trivial curves. This refers to the possibility that the family does not vary in moduli. Such behavior is already apparent in the failure of the function field analogue of the Mordell conjecture to hold in general. For example, the surface $C \times C$ admits infinitely many sections over $C$, given by the graphs of iterates of the Frobenius morphism. In this work, we formulate and prove analogues of the results in \cite{janbazi2025unified} that address these complications.

Let $C$ be a smooth affine curve over the finite field $\F_q$, and $\pi: S \rightarrow C$ be a smooth and proper morphism whose generic fiber is a smooth projective curve of genus $g$. Let $C^0$ be the set of closed points of $C$, and for any $\pp \in C^0$, let $S_\pp$ be the fiber of $S$ at $\pp$. We denote by $\Aut_C(S)$ the group of automorphisms of $S$ over the base $C$.

\begin{definition} 
We define the set $\Omega_{n,C}(S/C)$ as follows:
\begin{equation*}
    \Omega_{n,C}(S/C) = \left\{ A \subset S(C) :  
    \begin{array}{ll} 
        &\#A = n \\  
         
        &\#\, S_\pp \cap A = n \quad \,\forall \mathfrak{p} \in C^0  
    \end{array} 
    \right\}.
\end{equation*}

Also, let $\Omega_n(S/C)$ denote the set of horizontal divisors of degree $n$ on the surface $S$ that are \'etale over the base $C$; that is,
\[
\Omega_n(S/C)=\{D\subset S\mid D\to C \textrm{ is \'etale of degree } n\}.
\]
\end{definition}

Note that each element of $\Omega_{n,C}(S/C)$ corresponds to a horizontal divisor of degree $n$ that is étale over $C$. We emphasize that such a divisor arises from a collection of $n$ disjoint sections from $C$ to $S$. In particular, we have an inclusion $\Omega_{n,C}(S/C)\subset\Omega_n(S/C)$.

When the fibered surface $S$ is trivial over $C$, the relative Frobenius morphism $F_{S/C}$ acts on $S$ itself and hence induces an action on $\Omega_n(S/C)$. This action can generate infinitely many horizontal divisors that are not equivalent under automorphisms in $\Aut_C(S)$. For instance, let $D$ be a smooth projective curve defined over $\F_q$ and suppose there is a non-constant morphism $t: C \rightarrow D$. Define $S := D \times C$, so that $S$ is a trivial family over $C$. Then, for each $i \in \mathbb{N}$, the graph of the composition $F_D^i \circ t : C \rightarrow D$ defines a section from $C$ to $S$, which corresponds to an element of $\Omega_1(S/C)$. If we further assume that the genus of $D$ is at least $2$, then the automorphism group $\Aut_C(S)$ is finite, and so these sections define infinitely many distinct orbits under the action of $\Aut_C(S)$. Therefore, when defining an equivalence relation on $\Omega_n(S/C)$, we must account for the action of the relative Frobenius.

We now define the equivalence relation on $\Omega_n(S/C)$, depending on the geometry of $S$.

\begin{definition}\label{def equivalence}
 Let \(A, B \in \Omega_{n}(S/C)\). If $S\to C$ is:

\begin{description}[align=left, labelsep=0.5em, labelwidth=*, leftmargin=0pt]
    \item[trivial:]
    Then, the relative Frobenius $F := F_{S/C}$ is a $C$-morphism from $S$ to itself, and we define the equivalence:
    \begin{equation*}
        A \sim B \iff \exists \, r, s \in \mathbb{Z}_{\geq 0}, \, \psi \in \Aut_C(S) 
        \quad \psi \cdot F^r A = F^s B
    \end{equation*}

    \item[isotrivial:]   
    Here, \(A \sim B\) if and only if there exist integers \(r, s \in \mathbb{Z}_{\ge 0}\) and a \(C\)-isomorphism 
    $\psi: S^{(q^r)} \to S^{(q^s)}$ such that
    \begin{equation*}
        \psi \cdot F_{S/C}^{r} A = F_{S/C}^{s} B
    \end{equation*}
    
    \item[Non-isotrivial:]
    Then, the equivalence is defined by
    \begin{equation*}
        A \sim B \iff \exists\, \psi \in \Aut_C(S) \quad \psi \cdot A = B
    \end{equation*}   
\end{description}
\end{definition}

Further details about the relative Frobenius morphism and its properties can be found in Section~\ref{background}, and a more in-depth discussion of the geometry of the morphism 
$S\ra C$ is provided in Section~4. Now we can state our theorems.

\begin{theorem}\label{section theo}
    The set $\Omega_{n,C}(S/C)/\sim$ is finite.
\end{theorem}
More generally, we prove:
\begin{theorem}\label{etale theo}
    The set $\Omega_n(S/C)/\sim$ is finite.
\end{theorem}

\subsection{Methods of Proof}
The idea is to show that the proof of Theorem \ref{etale theo} reduces to Theorem \ref{section theo}. The reduction proceeds as follows. We first establish that if Theorem \ref{etale theo} holds for a base change $S \times_C C' \to C'$ along a finite generically Galois morphism $C' \to C$, then it also holds for $S \to C$. This descent argument enables us to replace $C$ with a convenient extension when needed.

In particular, we show that there exists a finite generically Galois extension $E \to C$ such that every divisor in the set $\Omega_n(S/C)$ pulls back to a union of sections over $E$. Furthermore, for isotrivial families, we show that there exists a finite generically Galois cover $C' \to C$ such that the base change $S \times_C C' \to C'$ becomes trivial.

Finally, it remains to verify Theorem \ref{section theo} in both the trivial and non-isotrivial settings. In the non-isotrivial case, the proof follows a similar strategy to \cite{janbazi2025unified}, while in the trivial case we analyze morphisms between curves and control the behavior of Frobenius.
The special case where the fibers have genus $g=0$ and $n=1$ is straightforward and will be treated separately.  We omit it from the main body of the proof for clarity.

In order to capitalize on the equivalence relation in the isotrivial case, we need to demonstrate that there are many $C$-isomorphisms from $S^{(q^r)}\ra S^{(q^s)}$. We establish this by showing that the sequence of Frobenius twists of $S$ as illustrated in the following commutative diagram:
\[
\begin{tikzcd}
\dots \arrow[r]         & S^{(q^n)} \arrow[r] \arrow[d] & \dots \arrow[r]         & S^{(q^2)} \arrow[d] \arrow[r] & S^{(q)} \arrow[r] \arrow[d] & S \arrow[d] \\
\dots \arrow[r, "F_C"'] & C \arrow[r, "F_C"']           & \dots \arrow[r, "F_C"'] & C \arrow[r, "F_C"']           & C \arrow[r, "F_C"']         & C          
\end{tikzcd}
\]
forms a preperiodic tower of schemes over $C$. In particular, we show:
\begin{theorem}\label{preperiodic}
If the family $S\ra C$ is isotrivial, then  $\{S^{(q^n)}\ra C\}_{n\ge 0}$ is a preperiodic family of schemes over $C$; that is, there exist integers $d,N\in \mathbb{N}$ such that for any $m,n\ge N$, if $d$ divides $m-n$, then $S^{(q^m)}$ and $S^{(q^n)}$ are isomorphic  over $C$.
\end{theorem}

An important step in the proof of this theorem is the following result.

\begin{theorem}\label{f.m. conic bundle}
    There are finitely many conic bundles over $C$ up to isomorphism over $C$.
\end{theorem}

\subsection{Related Results}
In this section, we highlight several foundational results concerning curves over function fields. These theorems reflect a deep analogy with the arithmetic of function fields and number fields and serve as both conceptual and technical tools for our work.

We begin with a theorem of Samuel, which serves as a function field analogue of Faltings' Theorem.

\begin{theorem}\label{samuel theo}\cite[Thm $4$]{PMIHES_1966__29__55_0}
Let $k$ be an algebraically closed field of characteristic $p$ and $K$ be a function field of transcendence degree $1$ over $k$. Let $X$ be an absolute non-singular curve of genus $g \geq 2$ over $K$. If $X$ is non-isotrivial over $K$, then $\#\, X(K)< \infty$.
\end{theorem}

The next theorem, due to Voloch, resembles Siegel’s theorem for number fields.

\begin{theorem}\label{siegel like}\cite[Thm $5.3$]{voloch1990explicit}
Let $K $ be a function field in one variable over $\F_q$ and $M$ be a finite set of places of $K$. Let $E$ be an elliptic curve over $K$ with a non-constant $j$-invariant. Assume $f\in K(E)$ is a non-constant function defined over $K$. Then
\begin{equation*}
    \#\,\{ x\in E(K) |\, v(f(x))\geq 0 \quad \forall v \notin M\} < \infty.
\end{equation*}
\end{theorem}

The following two results are not directly used in our proofs, but we include them to emphasize the strength of the analogy between number fields and function fields. They reflect the finiteness phenomena that pervade the arithmetic geometry of curves: 

\begin{theorem}\cite[Thm $1$]{michael.rosen.2017} 
  Fix $B\in \R^{\geq 0}$. Then there are finitely many geometric and separable extensions $L/\F_q(T)$ in a fixed algebraic closure of $\F_q(T)$ for which the discriminant divisor $d_{L/\F_q(T)}$ has degree at most $B$.
\end{theorem}

\begin{theorem}\label{silv, units}\cite{Silverman_1984}
    Let $k$ be an algebraically closed field of characteristic $p$. Let $X$ be a smooth, projective curve of genus $g$ over $k$, and let $K$ be its function field. Let $M$ be a finite set of points in $X(k)$ and $\O_M$ be the ring of $M$-integers in $K$. If $u,v\in \O^{\times}_M$ such that $u+v=1$, then either $u,v \in k^\times$ or the separable degree of $u$ is at most $2g-2+|M|$.
\end{theorem}

Finally, we recall a result on the extension of automorphisms from the generic fiber to the entire family.

\begin{theorem}\label{ext aut}
    \cite[4, Prop 4.6]{silverman1994advanced}  Let $R$ be a Dedekind domain with fraction field $K$ and $C$ be a smooth projective curve of genus at least $1$. Let $\CC$ over $R$ be a minimal proper regular model for $C$ over $K$ and $\CC^0$ be the largest sub-scheme of $\CC$ which is smooth over $R$. Then every $K$-automorphism $\tau : C\rightarrow C$ of the generic fiber of $\CC$ extends to $R$-automorphisms
    \begin{equation*}
         \tilde{\tau} : \CC \to \CC \qquad \text{and} \qquad \tau^0 : \CC^0 \to \CC^0.
    \end{equation*}
\end{theorem}

This extension property is crucial when studying the global structure of automorphism groups and identifying when two rational sections are considered equivalent.

\subsection{Outline}
The structure of this paper is as follows. In Section 2, we present the necessary background and foundational results that support the main argument. This includes an overview of noncommutative cohomology, the theory of smooth proper morphisms of curves, and the behavior of the relative Frobenius morphism in families of curves. We also provide a brief discussion of Azumaya algebras and $\P^1$-bundles. Section 3 provides more technical results about curves that are used repeatedly in the proof of Theorem \ref{etale theo}.

Section 4 is devoted to the study of isotrivial families of curves. We focus on understanding the structure of such families, as well as the relative Frobenius morphism and its iterates.
In Section 5, we employ the method of Galois descent to reduce Theorem \ref{etale theo} to the simpler statement of Theorem \ref{section theo}, in the case where the family is either trivial or non-isotrivial. This descent argument uses the Galois action on covers of the base curve to control rational points and sections.
Finally, Section 6 completes the paper by proving Theorem~\ref{section theo} in both the trivial and non-isotrivial settings.
\section{Background and notations}\label{background}

\subsection{Function Fields}

The following results can be found in \cite{silverman2009arithmetic} and \cite{stichtenoth2009algebraic}.

\begin{theorem}
    Let $C$ and $D$ be smooth, proper and geometrically irreducible curves defined over a field $k$, and let $\phi : C \rightarrow D$ be a morphism defined over $k$. Then $\phi$ is either constant or surjective.
\end{theorem}

If $\phi : C \rightarrow D$ is a nonconstant morphism defined over $k$, then composition with $\phi$ induces an injection of function fields fixing $k$,
\begin{equation*}
    \phi^*: k(D) \rightarrow k(C), \quad \quad \phi^* f = f \circ \phi.
\end{equation*}

\begin{theorem}
    Let $C$ and $D$ be smooth, proper and geometrically irreducible curves defined over $k$. Then:
    \begin{itemize}
        \item If $\phi: C \rightarrow D$ is a nonconstant morphism defined over $k$, then $k(C)$ is a finite field extension of $\phi^*(k(D))$.
        \item Conversely, if $\iota : k(D) \rightarrow k(C)$ is $k$-algebra injection, then there exists a unique nonconstant morphism $\phi : C \rightarrow D$ defined over $k$ such that $\phi^* = \iota$.
        \item If $\mathbb{K} \subset k(C)$ is a subfield of finite index containing $k$, then there exists a smooth curve $C'$ over $k$, unique up to $k$-isomorphism, and a nonconstant morphism $\phi: C \rightarrow C'$ defined over $k$ such that $\phi^*(k(C')) = \mathbb{K}$.
    \end{itemize}
\end{theorem}

\begin{remark}
    As a consequence, for any smooth and geometrically irreducible curve $C$ over $k$, there exists a unique smooth projective curve $\overline{C}$ over $k$ together with an open embedding $\iota : C \hookrightarrow \overline{C}$.
\end{remark}

For a morphism $\phi : C \rightarrow D$ defined over $k$, if $\phi$ is constant, we define the degree of $\phi$ to be zero. Otherwise, $\phi$ is finite and we define its degree as
\begin{equation*}
    \deg \phi = [k(C) : \phi^*(k(D))].
\end{equation*}
We say that $\phi$ is separable, inseparable, or purely inseparable according to the corresponding property of the field extension $k(C) / \phi^*(k(D))$.

As will be explained in the next section, the absolute Frobenius morphism is an endomorphism of curves defined over finite fields $\mathbb{F}_q$.

\begin{lemma}\label{maps form}
    Let $f : C \to D$ be a morphism defined over $\mathbb{F}_q$. Then either $f$ is constant, meaning $f(C) = \{x\}$ for some fixed $x \in D(\mathbb{F}_q)$, or $f$ can be factored as
    \[
        f = F_D^r \circ t,
    \]
    where $F_D$ is the absolute Frobenius on $D$, $r \geq 0$ is an integer, and $t: C \to D$ is a separable morphism.
\end{lemma}

\begin{theorem}[Riemann–Hurwitz formula]\label{Rim-Hur formula}
    Let $k$ be a field. Let $f: C \to D$ be a finite, separable morphism of smooth and geometrically irreducible curves over $k$. Let $d$ be the degree of $f$ and $R$ be the ramification divisor of $f$. Then
    \[
        -\chi(C) = -d \cdot \chi(D) + \deg(R),
    \]
    where $\chi(\cdot)$ denotes the Euler characteristic.
\end{theorem}

In the first part of this section, the curves $C$ and $D$ were assumed to be defined over the same field $k$. We now consider a more general situation where this is not necessarily the case. Specifically, let $\phi: D \to C$ be a finite separable morphism between smooth, proper and geometrically irreducible curves $C$ and $D$ defined over possibly different fields $k$ and $k'$, with function fields $K = k(C)$ and $L = k'(D)$ respectively.

\begin{theorem}
    Define $E = K k'$, the compositum of $K$ and $k'$. Then:
    \begin{itemize}
        \item The extension $k'|k$ is finite, and $K \cap k' = k$.
        \item We have
        \[
            [E : K] = [k' : k].
        \]
        \item The extension $L|E$ is an extension of function fields with the same constant field, and
        \[
            [L : K] = [L : E] \cdot [k' : k].
        \]
    \end{itemize}
\end{theorem}

This result about field extensions is particularly relevant when analyzing base changes, especially in the context of Galois descent.

\subsection{Relative Frobenius}

For any scheme $X$ over the finite field $\F_q$, there is an absolute Frobenius $F_X$ such that:
\begin{equation*}
    \begin{tikzcd}
X \arrow[rr, "F_X"] \arrow[rd] &      & X \arrow[ld] \\
                               & \F_q &             
\end{tikzcd}
\end{equation*}
For any schemes $X,Y$ over $\F_q$, we have $F_{X\times Y} = F_X \times F_Y$ and for any map $f:X\rightarrow Y$, we have 
\begin{equation*}
    \begin{tikzcd}
X \arrow[rr, "F_X"] \arrow[d, "f"'] &  & X \arrow[d, "f"] \\
Y \arrow[rr, "F_Y"]                &  & Y               
\end{tikzcd}
\end{equation*}
Let $k\ge 1$ be an integer and $X':=X\times_{\F_q} \F_{q^k}$, then 
\( F_{X'} \) is just $F_X^k$, the absolute Frobenius of $F_X$ composed with itself $k$ times.

Let $X $ be a scheme over the base $T$ over the finite field $\F_q$. Define $X^{(q/T)}$ to be the fiber product of the maps $F_T:T\rightarrow T$ and $f:X\rightarrow T$. By the previous comment we see that $f\circ F_X = F_T\circ f$, and therefore, by the universal property of the fiber product there exists a unique map $F_{X/T}:X\rightarrow X^{(q/T)}$. This map is called the relative Frobenius of $X$ over $T$.

\begin{equation*}
    \begin{tikzcd}
X \arrow[rrd, "F_X", bend left] \arrow[rdd, "f"', bend right] \arrow[rd, "F_{X/T}", dashed] &                             &                  \\
                                                                                            & X^{(q/T)} \arrow[r] \arrow[d] & X \arrow[d, "f"] \\
                                                                                            & T \arrow[r, "F_T"]          & T               
\end{tikzcd}
\end{equation*}
Also, for a map $T'\rightarrow T$ defined over $\F_q$, we have a natural isomorphism 
\begin{align*}
    (X\times_T T')^{(q/T')}&\cong X^{(q/T)}\times_T T'\\
    F_{X\times_T T' /T'} &=  F_{X/T} \times \mathbf{1}_{T'}
\end{align*}
Similarly, applying the same procedure with $F_T^m$ yields $X^{(q^m/T)}$ and the relative Frobenius $F_{X/T}^m$.
\begin{equation*}
    \begin{tikzcd}
X \arrow[rrd, "F_X^m", bend left] \arrow[rdd, "f"', bend right] \arrow[rd, "F_{X/T}^m", dashed] &                             &                  \\
                                                                                            & X^{(q^m/T)} \arrow[r] \arrow[d] & X \arrow[d, "f"] \\
                                                                                            & T \arrow[r, "F_T^m"]          & T               
\end{tikzcd}
\end{equation*}
When $T$ is understood, one can simply write $X^{(q^m)}$ instead of $X^{(q^m/T)}$.

Relative Frobenius is a universal homeomorphism. Thus, for any map $\alpha : X\rightarrow X $ of $T$-schemes there exists a unique map $\beta:X^{(q)}\rightarrow X^{(q)}$ of $T$-schemes such that
\begin{equation*}
 \begin{tikzcd}
X \arrow[rrr, "\alpha"] \arrow[dd, "F_{X/T}"'] &  &  & X \arrow[dd, "F_{X/T}"] \\
                                               &  &  &                         \\
X^{(q)} \arrow[rrr, "\beta"]                   &  &  & X^{(q)}                
\end{tikzcd}
\end{equation*}

In our setting, when the surface $S$ is isomorphic to $D\times _{\F_q} C$ for some smooth curve $D$ over $\F_q$ we have :

\begin{align*}
S^{(q/C)}\cong (D\times_{\F_q} C)^{(q/C)} &\cong D^{(q/\F_q)} \times_{\F_q} C \cong D\times_{\F_q} C  \\
F_{S/C}=F_{D\times_{\F_q} C  / C } &= F_{D/\F_q \times \mathbf{1}_C}
\end{align*}

So the relative Frobenius acts on the surface $S$ itself and if $\alpha :S\rightarrow S$ is an element of $\Aut_C(S)$, then there exists a map $\beta:S\rightarrow S$ such that the following diagram commutes.
\begin{equation}
\begin{tikzcd}
S \arrow[rrr, "\alpha"] \arrow[dd, "F_{S/C}"'] &  &  & S \arrow[dd, "F_{S/C}"] \\
                                               &  &  &                         \\
S \arrow[rrr, "\beta"]                         &  &  & S                      
\end{tikzcd}
\end{equation}

We refer to $\beta$ as the action of the Frobenius on the automorphism $\alpha$ and write it as $\alpha^F$. Note that $\mathbf{1}_S^F $ is just $\mathbf{1}_S $, so if for some $A,B\in \Omega_n(S/C)$ we have $A\sim B$, then $F\cdot A\sim F\cdot B$.

\subsection{Projective Space Over Dedekind Domains}\label{Proj space pver dede}

We recall the structure of $\mathbb{P}^1$-bundles over Dedekind domains, which is crucial for understanding models of genus $0$ curves. When the generic fiber of our surface has genus $0$ and contains at least one rational point, it is isomorphic to $\mathbb{P}^1$. In this case, the surface itself is a $\mathbb{P}^1$-bundle over the affine base curve.

Let $R$ be a Dedekind domain and set $X = \mathrm{Spec}\,R$. Suppose $Y$ is a $\mathbb{P}^1$-bundle over $X$, i.e., we have a morphism $\pi: Y \to X$ whose fibers are isomorphic to $\mathbb{P}^1$. Then $Y$ is the projectivization of a rank $2$ vector bundle over $X$. Two such vector bundles differing by tensoring with a line bundle produce isomorphic projectivizations.

Since $X$ is affine, every rank $2$ vector bundle corresponds to a torsion-free $R$-module $M$ of rank $2$. By the structure theorem for finitely generated modules over Dedekind domains, $M$ decomposes as
\[
M \cong I \oplus J
\]
for some fractional ideals $I, J$ of $R$. Because $Y$ is determined up to line bundle twists, we may, without loss of generality, take $J = R$. Furthermore, if $I$ and $I'$ represent the same class in the ideal class group of $R$, then the corresponding $\mathbb{P}^1$-bundles are isomorphic.

Let $I$ be an integral $R$-ideal and let $S$ be the projectivization of  $R\, \oplus\, I$ over $R$. Then the automorphism group of $S$ over $ R$ is given by 

\begin{equation*}
    \operatorname{Aut}_R(S) \cong \Biggl\{ \begin{pmatrix} \alpha & \beta \\ \gamma & \delta \end{pmatrix} : \alpha,\delta\in R,\ \beta\in I^{-1},\ \gamma\in I,\ \alpha\delta-\beta\gamma\in R^\times \Biggr\}.
\end{equation*}

Let $K $ be the fraction field of $R$. This automorphism group acts on $\P^1(K)$ given by
\[
\begin{pmatrix} \alpha & \beta \\ \gamma & \delta \end{pmatrix} \cdot [a:b] = [\alpha a+\beta b: \gamma a+\delta b]
\]
for a point $[a:b]\in \P^1(K)$.
\begin{lemma}\label{annoying case}
    There is an injection 
    \[
     \operatorname{Aut}_R(S) \backslash \P^1(K)  \xhookrightarrow{\hspace{1cm}}\Cl(R)
    \]
    where $\Cl(R)$ is the ideal class group of $R$.
\end{lemma}
\begin{proof}
   This proof is inspired by \cite{conradSL2}. There is a well-defined map $f:\P^1(K) \ra \Cl(R)$ such that the element $[a:b]\in\P^1(K)$ maps to the fractional ideal $aI+bR$ in the class group. Let $\phi=\begin{pmatrix} x & y \\ z& t \end{pmatrix} $ be an element of $\Aut_R(S)$. Then 
\[
\phi \cdot [a:b] = [ax+by:az+bt].  
\]
Moreover, $xt-yz$ is a unit in $R$, so $xR+zI^{-1}= yI+tR=R$ and $(ax+by)I+(az+bt)R =aI+bR$ as fractional ideals.
Therefore, $f$ induces a map from the orbits $ \operatorname{Aut}_R(S) \backslash \P^1(K)$ to $\Cl(R)$. 

Assume $f([a:b]) = f([a':b'])$. Then there exists $\lambda \in K$ such that $J := aI+bR = \lambda(a'I+b'R)$. Now we get
\[
R= (aI+bR)J^{-1} = aIJ^{-1} +bJ^{-1}
\]
and there exist $v\in IJ^{-1}$ and $u\in J^{-1}$ such that $av+bu=1$. Similarly, there exist $v'\in IJ^{-1}$ and $u'\in J^{-1}$ such that $\lambda a'v'+\lambda
b'u'=1$.

The matrices
\[
M=\begin{pmatrix} a & -u \\ b& v \end{pmatrix} \quad\quad  N=\begin{pmatrix}\lambda a'& -u' \\ \lambda b'& v' \end{pmatrix}
\]
have determinant $1$ and the matrix
\[
MN^{-1} = \begin{pmatrix} 
av'+\lambda b'u 
& au'-\lambda a'u 
\\ bv'-\lambda b'v
& bu'+\lambda a' v 
\end{pmatrix}
\]
 maps $[a':b']$ to $[a:b]$. Moreover, it has determinant $1$ and
 \[
 av'+\lambda b'u,\,  bu'+\lambda a' v\in R, \quad au'-\lambda a'u\in I^{-1},\quad bv'-\lambda b'v\in I
 \]
Hence, $MN^{-1}$ is an element of $\Aut_R(S)$, which implies that $f$ induces an injection 
\[
 \operatorname{Aut}_R(S) \backslash \P^1(K)  \xhookrightarrow{\hspace{1cm}}\Cl(R)
\]
\end{proof}

\subsection{Conic Bundles and Azumaya Algebras}\label{conic bundle sections}

This section follows \cite{milne1980etale} and \cite{de2006separable}.

\begin{definition}(Azumaya algebras)
\begin{itemize}
    \item Let $R$ be a commutative local ring. Let $A$ be an algebra over $R$ with an identity element $1$ such that the map $R\ra A$, $r\mapsto r\cdot 1$, identifies $R$ with a subring of the center of $A$. Let $A^{\mathrm{op}}$  denote the opposite algebra to $A$, that is, the algebra with the multiplication reversed. $A$ is an Azumaya algebra over $R$ if it is free of finite rank as a $R$-module and if the map $A\otimes_R A^{\mathrm{op}} \ra \text{End}_{R}(A)$ that sends $a\otimes a'$ to an endomorphism $(x\mapsto axa')$ is an isomorphism.
\item Let $X$ be a scheme. An $\mathcal{O}_X$-algebra $A$ is an Azumaya algebra over $X$ if it is coherent as an $\O_X$-module and if, for all closed points $x$ of $X$, $A_x$ is an Azumaya algebra over the local ring $\O_{X,x}$  
\end{itemize}
\end{definition}

Two Azumaya algebras $A$ and $A'$ over $X$ are said to be \emph{similar} if there exist locally free $\O_X$-modules $E$ and $E'$, of finite rank over $\O_X$, such that there is an isomorphism 
\[
A\otimes_{\O_X} \operatorname{End}_{\O_X}(E) \approx A'\otimes_{\O_X} \operatorname{End}_{\O_X}(E')
\]
The similarity relation is an equivalence relation, because $\operatorname{End}(E) \cong \operatorname{End}(E \otimes E')$. The tensor product of two Azumaya algebras is an Azumaya algebra, and this operation is compatible with the similarity relation. The set of similarity classes of Azumaya algebras on $X$ becomes a group under the operation $[A][A'] = [A \otimes A']$, the identity element is $[\mathcal{O}_X]$, and $[A]^{-1} = [A^{\mathrm{op}}]$. This is the \emph{Brauer group} $\operatorname{Br}(X)$ of $X$. 
Also, the \emph{cohomological Brauer group } of \( X \) is defined by
\[
\mathrm{Br}'(X) := H^2_{\text{\'et}}(X, \mathbb{G}_m).
\]

A \emph{conic bundle} over a scheme $X$ is a scheme $\pi: S \to X$ that is étale locally isomorphic to the projective line; that is, there exists an étale cover $\{U_i \to X\}$ such that
\[
S \times_X U_i \cong \mathbb{P}^1_{U_i} := \mathbb{P}^1 \times U_i
\]
as $U_i$-schemes. In other words, $S$ is a form of $\mathbb{P}^1$ over $X$ in the étale topology.

\begin{lemma}
    Isomorphism classes of conic bundles over a scheme $X$ are classified by the cohomology set $H^1_{\text{\'et}}(X, \PGL_2)$.
\end{lemma}
\begin{proof}

Since the automorphism group of $\mathbb{P}^1$ over a field is $\PGL_2$, the transition functions used to glue the local trivializations are valued in $\PGL_2$. Therefore, conic bundles over $X$ correspond to $\PGL_2$-torsors on $X$ for the étale topology. Such torsors are classified up to isomorphism by the nonabelian cohomology set $H^1_{\text{\'et}}(X, \PGL_2)$. 
Thus, the set of isomorphism classes of conic bundles over $X$ is in bijection with $H^1_{\text{\'et}}(X, \PGL_2)$.
    
\end{proof}

\begin{lemma}\cite[IV, Thm $2.5$, Step $1$]{milne1980etale}
    The set of isomorphism classes of Azumaya algebras of rank $n^2$ over $X$ is in bijection with $H^1_{\text{\'et}}(X,\PGL_n)$. 
\end{lemma}

By this lemma, studying conic bundles over a scheme $X$ is closely related to studying Azumaya algebras of rank $4$ over $X$, since both are classified by the same cohomology set $H^1_{\text{\'et}}(X, \PGL_2)$. 

\begin{theorem}\cite[IV, Thm $2.5$]{milne1980etale}
    There is a canonical injective homomorphism
    \[
    \mathrm{Br}(X) \to H^2_{\text{\'et}}(X, \mathbb{G}_m).
    \]
\end{theorem}

The first step in classifying Azumaya algebras over a local ring \( R \) is to fix the generic isomorphism class, meaning we consider only those orders that lie in a single Azumaya algebra \( \Sigma \) over $K$. The following powerful result shows that, within such a class, the structure is rigid: all Azumaya \( R \)-orders are conjugate.

\begin{theorem}\cite[\MakeUppercase{\romannumeral 5}, Lem $2.7$]{de2006separable}\label{azumaya over local ring}
    Let \( R \) be a local principal ideal domain with quotient field \( K \), and let \( \Sigma \) be an Azumaya algebra over \( K \). If \( A \) and \( B \) are \( R \)-orders in \( \Sigma \) with \( B \) an Azumaya algebra over \( R \), then there exists an invertible element \( t \in \Sigma \) with
    \[
    A \subseteq t^{-1} B t.
    \]
    In particular, \( B \) is a maximal order.\footnote{The authors use the term central separable for what we call Azumaya}
\end{theorem}

\subsection{Noncommutative Cohomology}

For noncommutative group cohomology, we closely adhere to Chapter 27 of Milne's notes on algebraic groups, as presented in \cite{milne2014algebraic}.

\begin{definition}
    Let $G$ be a group. A $G$-group $A$ is a group $A$ with an action 
    \begin{equation*}
        (\sigma, a) \mapsto \sigma a : G \times A \rightarrow A
    \end{equation*}
of $G $ on the group $A$.
\end{definition}

 Let $A$ be a $G$-group. Then  $H^0(G,A) :=A^G$, the set of elements in $A$ fixed under the action of $G$, i.e.,
    \begin{equation*}
        H^0(G,A)=A^G=\{a\in A \mid \sigma a =a \,\,\,\, \forall \sigma \in G\}.
    \end{equation*}

\begin{definition}
    Let $A$ be a $G$-group. Define $Z(G,A)$, the set of $1$-cocycles, as follows:
    \begin{equation*}
        Z(G, A) = \{ f: G \rightarrow A | f(\sigma \tau ) = f(\sigma) \cdot \sigma f(\tau) \,\,\,\,\,\, \forall \sigma , \tau \in G \}.
    \end{equation*}    
\end{definition}
Two $1$-cocycles $f, g \in Z(G,A)$ are equivalent if there exists $c\in A$ such that 
 \begin{equation*}
     g(\sigma) = c^{-1}\cdot  f(\sigma)\cdot  \sigma c \,\,\,\,\,\,\,\,\,\,\,\,\,\forall \sigma \in G.
 \end{equation*}
This is an equivalence relation on the set of $1$-cocycles, and $H^1(G,A)$ is defined to be the set of equivalence classes of $1$-cocycles. In general, $H^1(G,A)$  is not a group unless $A$ is commutative, but it has a distinguished element, namely, the class of $1$-cocycles of the form $\sigma \mapsto b^{-1}\cdot \sigma b, \, b\in A$ (the principal $1$-cocycle).

When $  A$ is commutative, $H^i(G,A)$ coincides with the usual cohomology groups for
$i=0, 1$.

\begin{theorem}
    An exact sequence 
    \begin{equation}\label{exact seq}
       \begin{tikzcd}
    1 \arrow[r] & A \arrow[r, "u"] & B \arrow[r, "v"] & C \arrow[r] & 1
\end{tikzcd}
    \end{equation}
    of $G$-groups gives rise to an exact sequence of pointed sets
    \begin{equation*}
        \begin{tikzcd}
1 \arrow[r] & {H^0(G,A)} \arrow[r, "u^0"] & {H^0(G,B)} \arrow[r, "v^0"] & {H^0(G,C)} \arrow[r, "\delta"] & {H^1(G,A)} \arrow[r, "u^1"] & {H^1(G,B)} \arrow[r, "v^1"] & {H^1(G,C)}.
\end{tikzcd}
    \end{equation*}
    More precisely:
    \begin{itemize}
        \item The sequence $  \begin{tikzcd}
1 \arrow[r] & {H^0(G,A)} \arrow[r, "u^0"] & {H^0(G,B)} \arrow[r, "v^0"] & {H^0(G,C)} 
\end{tikzcd}$ is exact as a sequence of groups.
\item There is a natural action of $C^G$ on $H^1(G,A)$.
\item The map $\delta $ sends $c\in C^G$ to $1\cdot c$ where $1$ is the distinguished element of $H^1(G,A)$ 
\item The nonempty fibres of $u^1:H^1(G,A) \rightarrow H^1(G,B)$ are the orbits of $C^G$ on $H^1(G,A)$.

\item The kernel of $v^1$ is the quotient of $H^1(G,A)$ by the action of $C^G$.
 
\end{itemize}
\end{theorem}

For maps of pointed sets, the term "kernel" refers to the fiber over the distinguished element. This theorem describes only the fiber of $v^1$ that contains the class of principal $1$-cocycles. To describe the other fibers, we need to consider appropriate twists of the $G$ action. 

\begin{definition}
    Let $B $ be a $G$-group, and let $S$ be a $G$-set with a left action of $B$ compatible with the action of $G$. Let $f\in Z(G,B)$, and let $\prescript{}{f}{S}$ denote the set $S$ on which $G$ acts by 
    \begin{equation*}
        \sigma * s = f(\sigma)\cdot \sigma s \quad\quad\forall \sigma \in G.
    \end{equation*}
 We say $\prescript{}{f}{S}$ is obtained from $S$ by twisting by a $1$-cocycle $f$.
\end{definition}

Now consider an exact sequence (\ref{exact seq}), and let $f\in Z(G,B)$. The group $B$ acts on itself by inner automorphism leaving $A$ stable, and so we can twist (\ref{exact seq}) by $f$ to obtain an exact sequence 
\begin{equation*}
    \begin{tikzcd}
1 \arrow[r] & \prescript{}{f}{A} \arrow[r, "u"] & \prescript{}{f}{B} \arrow[r, "v"] & \prescript{}{f}{C} \arrow[r] & 1.
\end{tikzcd}
\end{equation*}
The next theorem describes the fiber of $v^1$ containing  $[f]\in H^1(G,B)$.

\begin{theorem}\label{coh theo}
    There is a commutative diagram
    \begin{equation}
        \begin{tikzcd}
{H^0(G,\prescript{}{f}{C} )} \arrow[r] & {H^1(G,\prescript{}{f}{A} )} \arrow[r] & {H^1(G,\prescript{}{f}{B} )} \arrow[d, "\cong"] \arrow[r] & {H^1(G,\prescript{}{f}{C} )} \arrow[d, "\cong"] \\
{H^0(G,C)} \arrow[r, "\delta"]         & {H^1(G,A)} \arrow[r, "u^1"]            & {H^1(G,B)} \arrow[r, "v^1"]                               & {H^1(G,C)}                                     
\end{tikzcd}
    \end{equation}
in which the vertical arrows map the distinguished elements in $H^1(G,\prescript{}{f}{B})$ and $H^1(G,\prescript{}{f}{C})$
    to the class of $[f]$ and $v^1([f])$.
\end{theorem}

\begin{theorem}\cite[$5.8$]{serre1997galois}\label{inflation restriction}
  Let $H$ be a normal subgroup of $G$ and let $A$ be a $G$-group. The map $H^1(G/H, A^H) \rightarrow H^1 (G, A)$ given by the pull-back along $G\rightarrow G/H$ is injective
\end{theorem}

For cohomology of sheaves of noncommutative groups, we will follow \cite[\MakeUppercase{\romannumeral 3}] {milne1980etale}.
Let $G$ be a sheaf of groups on $X$, written multiplicatively, and let $\mathcal{U} = (U_i \to X)_{i \in I}$ be a covering of $X$. A 1-cocycle for $\mathcal{U}$ with values in $G$ is a family $(g_{ij})_{i,j}$, $g_{ij} \in G(U_{ij})$, such that
\[
(g_{ij}|_{U_{ijk}})(g_{jk}|_{U_{ijk}}) = (g_{ik}|_{U_{ijk}}).
\]

Two cocycles $g$ and $g'$ are \emph{cohomologous} if there is a family $(h_i)_i$, $h_i \in G(U_i)$, such that $g'_{ij} = (h_i|_{U_{ij}})g_{ij}(h_j|_{U_{ij}})^{-1}$. This is an equivalence relation, and the set of cohomology classes is written $\hat{H}^1(\mathcal{U}/X, G)$. It is a set with a distinguished element $(g_{ij})$ where $g_{ij} = 1$ for all $i$ and $j$. 

The set ${H}^1_{\text{\'et}}(X, G)$ is defined to be $\varinjlim \hat{H}^1(\mathcal{U}/X, G)$ where the limit is taken over all \'etale coverings of $X$, and the set  ${H}^1_{\text{Zar}}(X, G)$ is defined to be $\varinjlim \hat{H}^1(\mathcal{U}/X, G)$ where the limit is taken over all Zariski open  coverings of $X$.
When $G$ is abelian, this definition coincides with the usual cohomology groups. 

\begin{prop}
To any exact sequence of sheaves of groups
\[
1 \xrightarrow{\quad\quad} G' \xrightarrow{\quad\quad} G \xrightarrow{\quad\quad} G'' \xrightarrow{\quad\quad} 1,
\]
there is an associated  exact sequence of pointed sets
\[
1 \xrightarrow{\quad} G'(X) \xrightarrow{\quad} G(X) \xrightarrow{\quad} G''(X) \xrightarrow{\quad} {H}^1_{\text{\'et}}(X, G') \xrightarrow{\quad} {H}^1_{\text{\'et}}(X, G) \xrightarrow{\quad} {H}^1_{\text{\'et}}(X, G'').
\]
\end{prop}

\section{More on curves}\label{more on curves}
In this section, we establish several results on smooth curves that will be used in the proof of the main theorem.

 \begin{lemma}\label{bounded intersection of graph lemma}
    Let $C$ and $D$ be complete, smooth curves over a field $k$ and $X:=C\times D$. Let $f$ and $g$ be two distinct non-constant maps from $C$ to $D$ of degrees $d_f$ and $d_g$, respectively. Let $\Gamma_f$ and $\Gamma_g$ denote the graphs of $f$ and $g$ inside $X$. Then the intersection number $\Gamma_f\cdot \Gamma_g$ is bounded by constants depending only on $g_D$, $d_f$, and $d_g$. More precisely,
    \begin{equation*}
        d_f+d_g -\sqrt{4d_fd_gg_D^2} \le \Gamma_f\cdot \Gamma_g \le d_f+d_g +\sqrt{4d_fd_gg_D^2}.
    \end{equation*}  
\end{lemma}
\begin{proof}
    The proof uses several results from \cite[V.1]{hartshorne2013algebraic}. 
    Let $l=C\times \text{\{pt\}}$ and $m=\text{\{pt\}} \times D$. Then, the canonical divisor of $X$ is 
    \[K= (2g_D-2)l+(2g_C-2)m\] 
    where $g_C$ and $g_D$ denote the geometric genus of the curves $C$ and $D$. 
    Note that the graphs $\Gamma_f$ and $\Gamma_g$ are each isomorphic to $C$ and therefore have genus $g_C$. By the adjunction formula, we have
\begin{align*}
     \Gamma_f\cdot (\Gamma_f+K) &= 2g_C-2\\
     \implies \quad\Gamma_f\cdot \Gamma_f= 2g_C-2-\Gamma_f\cdot K = 2g_C-2-\big[ & d_f(2g_D-2) +2g_C-2 \big]=-d_f(2g_D-2).
\end{align*}
Similarly, we find 
\begin{equation*}
    \Gamma_g\cdot \Gamma_g= - d_g(2g_D-2). 
\end{equation*}
Let $k$ be an integer. Then 
\begin{align*}
    &   \quad \quad \quad \quad
    (\Gamma_f-k\Gamma_g)^2  \le 
    2((\Gamma_f-k\Gamma_g)\cdot l)((\Gamma_f-k\Gamma_g)\cdot m)\\
    &\implies\quad  \Gamma_f^2-2k\Gamma_f\cdot \Gamma_g + k^2\Gamma_g^2\le 2k^2d_g-2kd_f-2kd_g+2d_f\\
&\implies \quad \,\,\,\,\,\,
   0\le 2d_gg_Dk^2 + 2k(\Gamma_f\cdot \Gamma_g -d_f -d_g)+ 2d_fg_D
\end{align*}
If $g_D=0$, then evaluating the inequality at $k=\pm1$ gives the exact value 
\begin{equation*}
    d_f+d_g= \Gamma_f\cdot \Gamma_g.
\end{equation*}
If $g_D\ge 1$, then the quadratic polynomial above is positive definite, which implies
\begin{align*}
    (\Gamma_f\cdot \Gamma_g -d_f -d_g)^2&\le4d_fd_gg_D^2\\
    \implies\quad\quad d_f+d_g -\sqrt{4d_fd_gg_D^2} \le \Gamma_f\cdot \Gamma_g &\le d_f+d_g +\sqrt{4d_fd_gg_D^2}. 
\end{align*}
This completes the proof.
\end{proof}

\begin{lemma}\label{f. m. sep map}
   Let $C$ and $D$ be complete, smooth curves over $\F_q$. Then there are only finitely many separable morphisms $t:C \to D$ of fixed degree $d$.
\end{lemma}
\begin{proof}
Let $W$ denote the set of separable morphisms of degree $d$ from $C$ to $D$. By Theorem \ref{Rim-Hur formula}, for any $f \in W$,
\begin{equation*}
    2g_C - 2 = d(2g_D - 2) + \deg R_f,
\end{equation*}
where $R_f$ is the ramification divisor of $f$.

For any positive integer $n$, define:
\[
X_n = \bigcup_{i=1}^d C(\F_{q^{ni}}), \quad\quad Y_n = \left\{ a \in D(\overline{\F_q}) \,\middle|\, \kappa(a) = \F_{q^n} \right\},
\]
where $\kappa(a)$ denotes the residue field of $a$. Take \( n \) large enough so that:
\begin{align*}
    n &> \deg R_f, \\
    |X_n| &> d, \\
    |Y_n| &> 2d + 2d g_D.
\end{align*}

Since $f$ is defined over $\F_q$, it follows that $f^{-1}(Y_n) \subset X_n$. Moreover, every point in $Y_n$ has exactly $d$ preimages in $X_n$. Therefore, the number of possible such maps is at most
\[
\binom{|X_n|}{d}^{|Y_n|}.
\]

Furthermore, if two morphisms \( f, g \in W \) agree on all fibers over \( Y_n \), then their graphs \( \Gamma_f \) and \( \Gamma_g \) intersect in at least \( |Y_n| \) points, a number which is greater than the bound obtained in Lemma \ref{bounded intersection of graph lemma}. Since any such morphism $f$ is determined by its values on the fibers over $Y_n$, and for each $y \in Y_n$ there are $\binom{|X_n|}{d}$ possible choices for the fiber $f^{-1}(y)$, the total number of such maps is bounded above.
\end{proof}

The following result ensures that we only need to consider a finite segment of the Frobenius chain when comparing two morphisms. If the Frobenius actions along the two maps were too unbalanced, they would eventually intersect. This finiteness result controls the discrepancy between the Frobenius chains and allows us to work uniformly.

\begin{proposition}\label{frob control}
    Let $t_1, t_2\colon C \to D$ be nonconstant morphisms of smooth curves over the finite field $\F_q$. Then there are only finitely many integers $r \in \mathbb{N}$ such that
    \begin{equation}\label{frob equation}
        F_D^r \circ t_2(x) \neq t_1(x) \quad \text{for all } x \in C(\overline{\F_q}).
    \end{equation}
\end{proposition}
\begin{proof}
Let $V$ be the set of all integers $r$ for which \eqref{frob equation} holds. Denote by $\Gamma$ and $\Gamma_r$ the graphs of $t_1$ and $F_D^r \circ t_2$, respectively, inside $X := \overline{C} \times D$. Let $d_1 := \deg(t_1)$ and $d_2 := \deg(t_2)$, so the degree of $F_D^r \circ t_2$ is $d_r := q^r d_2$.

By Lemma \ref{bounded intersection of graph lemma}, we have:
\[
\Gamma \cdot \Gamma_r \geq q^r d_2 + d_1 - \sqrt{4 q^r g_D^2 d_1 d_2}
\]
where $g_D$ is the genus of $D$.
On the other hand, we may compute the intersection number via local intersection multiplicities:
\[
\Gamma \cdot \Gamma_r = \sum_{x \in \overline{C}(\overline{\F_q})} i_x(t_1, F_D^r \circ t_2),
\]
where $i_x(t_1, F_D^r \circ t_2)$ denotes the intersection multiplicity at $x$.

For any $r \in V$, the condition \eqref{frob equation} ensures that $i_x(t_1, F_D^r \circ t_2) = 0$ for all $x \in C(\overline{\F_q})$. Hence the sum is taken over a finite set of points:
\[
\Gamma \cdot \Gamma_r = \sum_{j=1}^k i_{x_j}(t_1, F_D^r \circ t_2),
\]
where $\{x_1, \dots, x_k\}$ is the finite complement of $C(\overline{\F_q})$ in $\overline{C}(\overline{\F_q})$.
 Let $z$ be a uniformizer at $x_j$, and suppose the local expansions of $t_1$ and $t_2$ at $x_j$ are
\begin{align*}
t_1(z) &= a_0 + a_l z^l + a_{l+1} z^{l+1} + \cdots = a_0 + \sum_{h=l}^\infty a_h z^h \quad \text{with } a_l \neq 0, \\
t_2(z) &= b_0 + b_1 z + b_2 z^2 + \cdots = \sum_{h=0}^\infty b_h z^h.
\end{align*}
Then the composition $F_D^r \circ t_2$ is given locally by
\[
F_D^r \circ t_2(z) = \sum_{h=0}^\infty b_h^{q^r} z^{q^r h}.
\]
When $q^r > l$, the intersection number $i_{x_j}(t_1, F_D^r \circ t_2)$ is at most $l$, since the smallest power of $z$ appearing in $t_1(z) - F_D^r \circ t_2(z)$ will be $l$ or lower. Hence, for large $r$, each such local intersection multiplicity is uniformly bounded, and the total $\Gamma \cdot \Gamma_r$ is bounded above independently of $r$.

But from the earlier inequality, we know that $\Gamma \cdot \Gamma_r$ grows without bound as $r \to \infty$. Therefore, only finitely many $r$ can satisfy \eqref{frob equation}, and so $V$ is finite.
\end{proof}

\begin{prop}\label{negative char implies finite aut}
    Let $k$ be an algebraically closed field and $C$ be a smooth quasi-projective curve over $k$. Assume $\chi(C)$, the Euler characteristic of $C$, is less than zero. Then the number of automorphisms of $C$ is finite.
\end{prop}
\begin{proof}
    Let $\ol C$ be the completion of $C$ and let $M:=\ol C \backslash C$. The set $M $ is finite and let $m $ be its size. Then $\chi (C)= 2-2g-m$, where $g$ is the genus of $C$. Any automorphism of $C$ extends to an automorphism of $\ol C$ that fixes the set $M$. Let $\Aut(C)$ and $\Aut_M(C)$ be the group of automorphisms of $C$ and the group of automorphisms of $C$ that extend to identity of $M$ respectively. Then $\Aut_M(C)$ is a subgroup of $\Aut(C)$ and its index is at most $m!$. Therefore, it is enough to show that $\Aut_M(C)$ is finite. Based on the genus, $g$, we will show $\Aut_M(C)$ is finite.
 \begin{itemize}
    \item $\boldsymbol{g=0:}$ Let \( f \in \Aut_M(C) \) be a nontrivial automorphism. Let \(\Gamma\) and \(\Gamma_f\) denote the graphs of the identity and of \(f\), respectively, considered as divisors in \(\overline{C} \times \overline{C}\). Since both maps have degree \(1\) and \(\overline{C}\) has genus \(0\), Lemma \ref{bounded intersection of graph lemma} implies that
\[
\Gamma \cdot \Gamma_f \leq 2.
\]
On the other hand, since \(f\) acts trivially on the marked points \(M\), the graphs \(\Gamma\) and \(\Gamma_f\) must intersect at least \(m\) times, and thus
\[
\Gamma \cdot \Gamma_f \geq m.
\]
Combining the inequalities yields \(m \leq 2\), which implies
\[
\chi(C) = 2 - 2g - m \geq 0.
\]
This contradicts the assumption that \(\chi(C) < 0\), and we conclude that \(\Aut_M(C)\) must be finite.

    \item$\boldsymbol{g=1:}$ By \cite[Prop $7.13$ and Rem $7.14$]{milne1996elliptic}, the number of automorphisms of $C$ that fix the origin $O$ is bounded by $24$. The Euler characteristic of $C$ is less than zero therefore $m\ge 1$. Without loss of generality we may assume $O\in M$. Then the number of elements in  $\Aut_M(C)$ is bounded by $24$ as well.

    \item $\boldsymbol{g\ge 2:}$ By \cite[\MakeUppercase{\romannumeral 4}]{hartshorne2013algebraic} if the characteristic of $k$ is zero, then the number of automorphisms of $\ol C$ is bounded by $84(g-1)$.
    By \cite{stichtenoth1973automorphismengruppe}, when the characteristic of $k$ is bigger than $0$, the number of automorphisms of $\ol C$ is bounded by \(16 \cdot g^4\)
unless $\ol C$ is the curve with equation
\[
y^{p^n} + y = x^{p^{n+1}},
\]
in which case it has genus $g = \frac{1}{2}p^n(p^n - 1)$ and the number of automorphisms is exactly 
\[
  p^{3n}(p^{3n} + 1)(p^{2n} - 1).
\]

Therefore, the number of automorphisms that fix the set $M$ is finite.
\end{itemize}
\end{proof}
\section{Geometry of $S/C$}\label{ Geometry}

Let \( C \) be a smooth affine curve over the finite field \( \mathbb{F}_q \), and let  
\[ \pi: S \xrightarrow{\quad\quad} C \]  
be a smooth and proper morphism whose generic fiber is a smooth, projective curve of genus \( g \).

\begin{definition}
We say that the surface \( S \) over \( C \) is:

\begin{itemize}
    \item \textbf{Trivial:} If there exists a smooth and projective curve \( D \) over \( \mathbb{F}_q \) such that  
    \[
    S \cong D \times_{\mathbb{F}_q} C
    \]
    as \( C \)-schemes.

    \item \textbf{Isotrivial:} If there exists a morphism of curves \( C' \to C \) and a curve \( D \) defined over the field of definition of \( C' \) such that  
    \[
    S \times_C C' \cong D \times C'
    \]
    as \( C' \)-schemes.

    \item \textbf{Non-isotrivial:} If it is not isotrivial.
\end{itemize}
\end{definition}

In this section, we study the properties of isotrivial surfaces and the effect of the relative Frobenius. Moreover, we prove Theorem \ref{f.m. conic bundle} and Theorem \ref{preperiodic}.

\begin{proposition}\label{trivial
after separable}
    If $ \pi: S \rightarrow C$ is isotrivial, then there exists a finite, generically Galois morphism $C' \rightarrow C $ such that 
    $ \pi_{C'} : S\times_C C' \rightarrow C'$ is trivial.
\end{proposition}
\begin{proof} If there exists a separable and finite $C'\ra C $ with the desired property, then we can take the Galois closure to obtain the desired morphism.

The paper \cite{moret2020points} shows that there exists a point $s \in  S$ such that the function field of $s$ is a separable extension of the function field of $ \pi(s)$. After normalizing $C$ through $s$, we obtain a multisection $T$ in $S$, that is separable and finite over $C$. Since a tower of separable extensions is also separable, without loss of generality, we may assume $ S(C) \neq \emptyset$.

The case of $g\ge 2 $ has been discussed in \cite[$4$]{creutz2024galois}, while the case of elliptic curves has been worked out in \cite[Lec $1$]{ulmer2011park}. For $g=0$, based on the discussion in Section \ref{Proj space pver dede}, it is enough to consider the base change to the Hilbert class field of $\F_q[C]$ where every ideal becomes principal.
\end{proof}

\subsection{Proof of Theorem \ref{f.m. conic bundle}}
 By \cite[IV, Cor $2.4$]{milne1980etale}, we have a long exact sequence of sheaves 
\begin{equation*}
   \begin{tikzcd}
0 \arrow[rr] &  & \G_m \arrow[rr] &  & \GL_2 \arrow[rr] &  & \PGL_2 \arrow[rr] &  & 0
\end{tikzcd}
\end{equation*}
on the \'etale site of $C$. By \cite[IV, Thm $2.5$, Step $3$]{milne1980etale}, there is a long exact sequence of pointed cohomology sets
\begin{equation*}
    \begin{tikzcd}
{H^1_{\text{\'et}}(C,\GL_2)} \arrow[rr] &  & {H^1_{\text{\'et}}(C,\PGL_2)} \arrow[rr,"\delta"] &  & {H^2_{\text{\'et}}(C,\G_m)}. 
\end{tikzcd}
\end{equation*}
Conic bundles over $C$ are classified by ${H^1_{\text{\'et}}(C,\PGL_2)}$. As mentioned in \cite[IV, Prop $2.7$]{milne1980etale}, the image of $\delta $ is a subset of $2$-torsion points in $H^2_{\text{\'et}}(C,\G_m)$. We will show that the map 
\begin{equation*}
    H^1_{\text{\'et}}(C,\PGL_2) \longrightarrow H^2_{\text{\'et}}(C,\G_m)
\end{equation*}
has finite image and finite fibers, implying that $ H^1_{\text{\'et}}(C,\PGL_2)$ is finite.

\begin{lemma}
    The group $H^2_{\text{\'et}}(C,\G_m)[2]$ is finite.
\end{lemma}
\begin{proof}
    By \cite[$2$, Prop $2.1$]{milne2006arithmetic}, we have an injection
   \begin{equation*}
    H^2_{\text{\'et}}(C,\G_m) \xhookrightarrow{\hspace*{1cm}} \bigoplus_{v\in \overline{C} \setminus C} \text{Br}(K_v),
\end{equation*}
    where the sum is over all closed points of $\ol C $ that are not in $C$. The Brauer group of any local field is isomorphic to $\Q/\Z$. Therefore, we have an injection 
    \begin{equation*}
         H^2_{\text{\'et}}(C,\G_m)[2] \xhookrightarrow{\hspace*{1cm}}\bigoplus_{v\in \ol C \backslash C} {\Z}/2\Z
    \end{equation*}
   Since the set of closed points $\overline{C} \setminus C$ is finite, and each $\operatorname{Br}(K_v)[2] \cong \Z/2\Z$, the direct sum $\bigoplus_{v\in \ol C \setminus C} \Z/2\Z$ is finite. This implies  $H^2_{\text{\'et}}(C,\G_m)[2]$ is a finite set.
\end{proof}

Now we will study the fibers of $\delta$.

\begin{lemma}
    Let $A$ and $B$ be Azumaya algebras over a scheme $C$, both of rank $4$, and suppose they define the same class in the cohomology group $H^2_{\text{\'et}}(C, \mathbb{G}_m)$. Then $A$ and $B$ are Zariski locally isomorphic; that is, for any closed point $x \in C$, there exists a Zariski open neighborhood $U$ of $x$ such that $A|_U \cong B|_U$ as Azumaya algebras over $U$.
\end{lemma}
\begin{proof} 
  Let $x\in C$ be a closed point, and write $R := \O_{C,x}$ for its local ring.  Since Zariski neighborhoods of $x$ correspond to localizations of $R$, it suffices to prove:
  \[
    A_x \;\cong\; B_x
  \]
  as $R$-algebras, where $A_x:=A\times_C \Spec R$.  Indeed, an isomorphism of finite locally free $R$-algebras extends to one on some open neighborhood of $x$.
  
The ring $R$ is a discrete valuation ring with maximal ideal $\mathfrak m$ and fraction field $K=\operatorname{Frac}(R)$. By hypothesis, $A_x$ and $B_x$ become isomorphic over the generic point of $\Spec R$  because they represent the same Brauer class.
     The Azumaya algebras $A_x$ and $B_x$ restrict to $R$-orders in the same central simple $K$-algebra
    \[
      \Sigma := A_x \otimes_R K \;\cong\; B_x \otimes_R K,
    \]
 Applying Theorem \ref{azumaya over local ring} to the central separable $R$-orders $A_x$ and $B_x$ in $\Sigma$, we obtain an invertible $t\in\Sigma^\times$ with
  \[
    A_x \cong t^{-1}B_x\,t.
  \]
  This completes the proof of the lemma.
\end{proof}

Let $A$ be an Azumaya algebra over $C$, and let $G$ be the sheaf of automorphisms of $A$ as an $\mathcal{O}_C$-algebra on the Zariski site of $C$. That is, $G$ assigns to each open subset $U \subseteq C$ the group of $\mathcal{O}_U$-algebra automorphisms of $A|_U$. Since all Azumaya algebras in the fiber over the Brauer class of $A$ are Zariski-locally isomorphic to $A$, they correspond to $G$-torsors in the Zariski topology. Thus, the fiber of $\delta$ over the class of $A$ injects into $H^1_{\text{Zar}}(C, G)$.

\begin{lemma}
   The number of elements of $H^1_{\text{Zar}}(C,G)$ is finite.
\end{lemma}
\begin{proof}
Let $K:=\F_q(C)$ be the function field of $C$ and let $j:\Spec K \ra C$ be the inclusion of the generic point. For any closed point $v$ of $C$, let $i_v$ be the inclusion of $v$ in $C$. Moreover, let $\O_v$ be the local ring at $v$ and $K_v$ be the fraction field of $\O_v$. Then we have a long exact sequence of sheaves on $C$
\[
\begin{tikzcd}
0 \arrow[r] & G \arrow[r]   & j_*G \arrow[r]   & \bigoplus_{v\in C} i_v^* G(K_v)/G(O_v) \arrow[r]   & 0
\end{tikzcd}
\]
By taking the Zariski cohomology, we get the following long exact sequence
\[
\begin{tikzcd}
G(K) \arrow[r] & \bigoplus_{v\in C}  G(K_v)/G(O_v) \arrow[r] & {H^1_{Zar}(C,G)} \arrow[r] & {H^1(C,j_*G)}
\end{tikzcd}
\]

Since \(j_*G\) is a flasque sheaf, and flasque sheaves are acyclic, we deduce that \(H^1(C,j_*G)=0\). Let $(a_v)_v$ and $(b_v)_v$ be two elements of $ \bigoplus_{v\in C}  G(K_v)/G(O_v)$ such that for some $g\in G(K)$ we have $b_v=ga_v$ for any $v\in C$. Let $M$ be the subset of all $v$ such that either $a_v\notin G(\O_v)$ or $b_v\notin G(\O_v)$. So, $M$ is a finite set $\{v_1,\dots v_m\}$. Let $U_0$ be the complement of $M$ in $C$ and for any $i=1\dots m$ let $U_i= U_0\cup \{v_i\}$. Then the image of $(a_v)_v$ is given by the $1$-cocycle $(c_{ij})$ for the cover $\{U_0,U_1 \dots , U_m\}$ by 
\begin{equation*}
    c_{ij} = \begin{cases}
        a_{v_i}\quad\quad i=0, j\neq 0
        \\
        a_{v_j}^{-1}\quad\quad i\neq 0 , j=0
        \\
        1 \quad\quad \,\,\,\text{Otherwise}
    \end{cases}
\end{equation*}
Similarly the image of $(b_v)_v$ is given by 
\begin{equation*}
    d_{ij} = \begin{cases}
        b_{v_i}\quad\quad i=0, j\neq 0
        \\
        b_{v_j}^{-1}\quad\quad i\neq 0 , j=0
        \\
        1 \quad\quad \,\,\,\text{Otherwise}
    \end{cases}
\end{equation*}
These two $1$-cocyles in $H^1_{Zar}(C,G)$ are equivalent by the $1$-coboundary $(h_i)$ given by
\begin{equation*}
    h_{i} = \begin{cases}
        g\quad\quad i=0
        \\
        1\quad\quad i\neq 0 
    \end{cases}
\end{equation*}
Therefore, we have shown that if two elements of  $\bigoplus_{v\in C}  G(K_v)/G(O_v)$ differ by an element of $G(K)$, then they have the same image in ${H^1_{Zar}(C,G)}$. Hence, it is enough to show that $G(K)\backslash \bigoplus_{v\in C}  G(K_v)/G(O_v)$ is finite. This is nothing but the class group of $G$ over the Dedekind domain $\F_q[C]$. This is finite due to the work of Conrad \cite[Thm $1.3.1$]{conrad2012finiteness}.
\end{proof}

\subsection{Proof of Theorem \ref{preperiodic}}
We break the proof based on $g$, the genus of the family $\pi:S\ra C$.

If $g=0$,  then all of the $\{ S^{(q^m)} \ra C \}$ are conic bundles over $C$. Theorem \ref{f.m. conic bundle} implies that there exist integers $m\neq n$ such that $S^{(q^m)}\cong S^{(q^n)}$ as $C$-schemes. Then by taking $d:=|\,m-n\,|$ and $N:= \min\{m,n\}$, Theorem \ref{preperiodic} follows. 
Now we address the case \( g \ge 1 \). Let \( L/K \) be a Galois extension of function fields with Galois group \( G \).

\begin{lemma}\label{automorphism has finite coh}
Let \( E \) be a smooth, complete curve over \( L \) of genus at least \( 1 \). Then
\[
\#\, H^1(G, \Aut(E)) < \infty.
\]
\end{lemma}

\begin{proof}
If \( g \ge 2 \), then the automorphism group \( \Aut(E) \) is finite, and the cardinality of \( H^1(G, \Aut(E)) \) is bounded by \( |\Aut(E)|^{|G|} \), hence finite.

Now suppose \( g = 1 \). Let \( H \) be a finite Galois extension of \( K \) with Galois group \( G' \), such that \( H \supseteq L \) and \( E \) has an \( H \)-rational point. Then the base change \( E' := E \times_L H \) is isomorphic to an elliptic curve with a distinguished origin \( O \). In this case, we have a short exact sequence:
\[
\begin{tikzcd}
1 \arrow[r] & E'(H) \arrow[r] & \Aut_H(E') \arrow[r] & \Aut_O(E') \arrow[r] & 1
\end{tikzcd}
\]
Here, \( \Aut_O(E') \) denotes the subgroup of automorphisms of \( E' \) that fix the origin \( O \). By Proposition \ref{negative char implies finite aut}, this group is finite. Furthermore, the group \( E'(H) \) is finitely generated and abelian, so \( H^1(G', E'(H)) \), as well as all its twisted forms, are finite by \cite[2, Cor $1.32$]{milne2011class}. Then, by Theorem \ref{coh theo}, the cohomology group \( H^1(G', \Aut_H(E')) \) is finite as well.

Let \( G'' \) denote the Galois group of $H/L$. Because \( L/K \) is also Galois, we have \( G'' \trianglelefteq G' \), and \( G'/G'' \cong G \). Then, by Theorem \ref{inflation restriction}, there is an injection
\[
H^1(G, \Aut(E)) \xhookrightarrow{\quad\quad} H^1(G', \Aut_H(E')).
\]
This completes the proof.
\end{proof}

Let $C'\ra C$ be the Galois extension obtained from Proposition \ref{trivial
after separable} and let $D$ be such that $S\times_C C' \cong D\times C'$. Let $G$ be its Galois group and $\F_{q^k}$ be the field of definition of $C'$. We have the following fiber product diagram 
\begin{equation*}
\begin{tikzcd}
C' \arrow[rr, "F_{C'}"] \arrow[d] &  & C' \arrow[d] \\
C \arrow[rr, "F_C^k"']            &  & C           
\end{tikzcd}
     \end{equation*}
and by considering the further fiber product with $\pi:S\ra C$, we have 

\[
\begin{tikzpicture}[commutative diagrams/every diagram]
\node (P1) at (0,0) {$C'$};
\node (P2) at (3,0) {$C'$} ;
\node (P3) at (3,3) {$D\times C'$};
\node (P4) at (0,3) {$D\times C'$};
\node (Q1) at (5, -2) {$C$};
\node (Q2) at (8, -2) {$C$};
\node (Q3) at (8, 1) {$S$};
\node (Q4) at (5, 1) {$S^{(q^k)}$};
\path[commutative diagrams/.cd, every arrow, every label]
(P1) edge node {$F_{C'}$}(P2)
(P4) edge[shorten >=3pt, shorten <=3pt] node {} (P3)
(P4) edge[shorten >=3pt, shorten <=3pt] node {}(P1)
(P3) edge[shorten >=3pt, shorten <=3pt] node {}(P2)
(P4) edge[preaction={draw=white, line width=2.5pt}] node {}(Q4)
(P1) edge node {}(Q1)
(P2) edge node {}(Q2)
(P3) edge node {}(Q3)
(Q1) edge node {$F_C^k$}(Q2)
(Q4) edge[shorten >=3pt, shorten <=3pt] node {} (Q3)
(Q3) edge[shorten >=3pt, shorten <=3pt] node {}(Q2)
(Q4) edge[shorten >=3pt, shorten <=3pt, preaction={draw=white, line width=2.5pt}] node {}(Q1);
\end{tikzpicture}\]

Therefore, all the $\{S^{(q^{ki})}\}_{i \in \mathbb{N}}$ are isotrivial families over $C$, all of which become isomorphic over $C'$ to a single family $D\times C'$. Hence, all $S^{(q^{ki})}$ correspond to a $1$-cocycle in $H^1(G, \Aut_{C'}(D\times C'))$. 

Let $E$ be the generic fiber of  $ \pi_{C'} : S\times_C C' \rightarrow C'$, which is isomorphic to the base change of 
$D$ to the function field of $C'$. Theorem \ref{ext aut} implies  $\Aut_{C'}(D\times C')\cong \Aut(E)$, and by  Lemma \ref{automorphism has finite coh} both cohomology groups $H^1(G, \Aut_{C'}(D\times C'))$ and $H^1(G, \Aut(E))$ are finite. Hence, there exists $i\neq j$ such that 
\begin{equation*}
    S^{(q^{ki})}\cong S^{(q^{kj})}
\end{equation*}
as schemes over $C$ and the family is preperiodic.

\begin{remark}
The conic bundle defined by the equation
\[
X^2 + tY^2 + Z^2 + tXY + tYZ + XZ = 0
\]
over the ring $\mathbb{F}_2[t]$ provides an example of a family of genus \( 0 \) curves that is preperiodic  but not periodic. 

This conic bundle is not trivial over $\mathbb{F}_2[t]$. In particular, it does not admit a section over $\Spec\,\mathbb{F}_2[t]$, and thus is not isomorphic to the projective line over this base. However, it becomes trivial after a finite base change: namely, over $\mathbb{F}_2[\sqrt{t}]$ the conic does admit a rational point, so it becomes isomorphic to $\mathbb{P}^1$ over $\Spec\,\mathbb{F}_2[\sqrt{t}]$.

Note that $\mathbb{F}_2[\sqrt{t}]$ has a trivial class group. By Section \ref{Proj space pver dede}, every conic with a rational point is isomorphic to $\mathbb{P}^1$. As a result, after one pull-back by Frobenius, the family becomes trivial and isomorphic to the product $\mathbb{P}^1 \times \Spec\,\mathbb{F}_2[\sqrt{t}]$.
\end{remark}

\section{Galois descent}\label{Descent}
In this section, we reduce Theorem 1.2 to Theorem 1.1 using Galois descent.

First, we show that for any such extension \( C' \to C \), there is a natural finite map of quotient sets:
\[
\Aut_C(S) \backslash \Omega_n(S/C) \longrightarrow \Aut_{C'}(S_{C'}) \backslash \Omega_n(S_{C'}/C'),
\]
where \( S_{C'} := S \times_C C' \). As a result, each equivalence class over \( C' \) intersects only finitely many classes over \( C \). Consequently, if Theorem \ref{etale theo} holds over \( C' \), it must also hold over \( C \).

We then show that there exists a finite, generically Galois extension \( E/C \) such that the elements of \( \Omega_n(S/C) \) map into \( \Omega_{n,E}(S_E/E) \) under base change. This reduces the proof of Theorem \ref{etale theo} to proving Theorem \ref{section theo}. Finally, by Proposition \ref{trivial after separable}, it suffices to establish Theorem \ref{section theo} in the trivial and non-isotrivial cases.
We exclude the case \( g = 0 \) and \( n = 1 \), as this case has already been treated for Theorems \ref{section theo} and \ref{etale theo} in Lemma \ref{annoying case}. The rest of this section is devoted to analyzing the remaining cases.

\begin{proposition}\label{prop that div becomes sections}
    There exists a finite, generically  Galois cover of curves \( E \to C \) such that for any smooth and proper morphism \( \pi: S \to C \) with geometrically connected generic fiber, we have
    \[
        \Omega_{n}(S/C) \xhookrightarrow{\quad\quad} \Omega_{n,E}(S_E/E),
    \]
    where \( S_E := S \times_C E \).
\end{proposition}
\begin{proof}
    Let \( A \in \Omega_n(S/C) \). Then \( A \) is a union of irreducible curves \( D_1, \dots, D_r \) inside \( S \). Since \( A \in \Omega_n(S/C) \), the restriction of \( \pi: S \to C \) to each \( D_i \) induces a finite étale map
    \[
        \pi_i: D_i \to C
    \]
    of degree \( a_i \) for each \( i = 1, \dots, r \). Each such \( D_i \) is defined over some finite extension \( \F_{q^j} \) where \( j \leq n \), and the map \( \pi_i \) is unramified. By the Riemann–Hurwitz formula, we have
    \[
        \chi(D_i) = a_i \chi(C).
    \]
    Moreover, \( \pi_i \) extends to a finite morphism \( \overline{D_i} \to \overline{C} \) of the same degree \( a_i \). The induced map
    \[
        \overline{D_i} \setminus D_i \to \overline{C} \setminus C
    \]
    is also finite of degree \( a_i \), so
    \[
        \#\, \overline{D_i}(\overline{\F_q}) \setminus D_i(\overline{\F_q}) \leq a_i \cdot \#\, \left( \overline{C}(\overline{\F_q}) \setminus C(\overline{\F_q}) \right).
    \]
    On the other hand, we have
    \[
        \chi(D_i) = 2 - 2g_{D_i} - \#\, \left( \overline{D_i}(\overline{\F_q}) \setminus D_i(\overline{\F_q}) \right).
    \]
    Since \( C \) is fixed and each \( a_i \in \{1, \dots, n\} \), it follows from the above equations that both \( \chi(D_i) \) and \( g_{D_i} \) are bounded. Over a finite field, there are only finitely many isomorphism classes of curves of bounded genus \cite{de2000counting}, and by Lemma \ref{f. m. sep map}, only finitely many separable maps from such curves to \( \overline{C} \).

    Taking the fiber products of all such maps and forming their Galois closure yields a finite Galois cover \( E \to C \) such that all the curves \( D_i \) become sections over \( E \). This completes the proof.
\end{proof}

Throughout the remainder of this section, we assume either $g\neq 0$ or $n\neq 1$.  Fix a finite, generically Galois morphism of smooth curves \( C' \to C \) with Galois group $G$, and let \( S' := S \times_C C' \). Note that we do not assume \( C' \) is defined over \( \F_q \). The property of being étale is preserved under base change. Hence, for any \( A \in \Omega_n(S/C) \), its pullback 
\( A' := A \times_C C' \subset S' \) lies in \( \Omega_n(S'/C') \). This yields a well-defined map
\[
    \rho: \Omega_n(S/C) \longrightarrow \Omega_n(S'/C').
\]

\begin{definition}
Let \( A \in \Omega_n(S/C) \). Define:
\begin{align*}
    \Sigma_A &:= \Aut_{C'}(S') \cdot \rho(A) \,\cap\, \rho\big(\Omega_n(S/C)\big), \\
    M_A &:= \{ \psi \in \Aut_{C'}(S') \mid \psi \cdot \rho(A) = \rho(A) \},
\end{align*}
where the group \( \Aut_{C'}(S') \) acts on \( \Omega_n(S'/C') \) via its action on \( S' \).
\end{definition}

\begin{lemma}\label{injective of multi orbit to coh}
    For any \( A \in \Omega_n(S/C) \), there is a natural injection
    \[
        \Aut_C(S) \backslash \Sigma_A \xhookrightarrow{\quad\quad} H^1(G, M_A).
    \]
\end{lemma}

\begin{proof}
To simplify notation, we do not distinguish between \( \rho(A) \) and \( A \). Let \( B \in \Sigma_A \), and suppose \( B = \psi \cdot A \) for some \( \psi \in \Aut_{C'}(S') \). Since both \( A \) and \( B \) are defined over \( K(C) \), for any \( \sigma \in G \), we have
\[
    B=\sigma(B) = \sigma(\psi \cdot A) = \sigma \psi \cdot \sigma A = \sigma \psi \cdot A,
\]
so that \( \psi^{-1} \sigma \psi \in M_A \). This defines a 1-cocycle
\[
    \Phi_{B,\psi}: G \xrightarrow{\quad} M_A, \quad \sigma \mapsto \psi^{-1} \sigma \psi.
\]

Now suppose \( B = \psi_1 \cdot A = \psi_2 \cdot A \). Then \( \psi_1^{-1} \psi_2 \in M_A \), so we compute:
\[
    \Phi_{B,\psi_2}(\sigma) = \psi_2^{-1} \sigma \psi_2 = (\psi_1^{-1} \psi_2)^{-1} \cdot \psi_1^{-1} \sigma \psi_1 \cdot \sigma(\psi_1^{-1} \psi_2),
\]
which shows \( \Phi_{B,\psi_1} \) and \( \Phi_{B,\psi_2} \) are cohomologous. Thus, the cohomology class \( [\Phi_{B,\psi}] \in H^1(G, M_A) \) only depends on \( B \in \Sigma_A \), and we may write \( \Phi_B := [\Phi_{B,\psi}] \).

We obtain a well-defined map:
\[
    \Phi: \Sigma_A \xrightarrow{\quad} H^1(G, M_A), \quad B \mapsto \Phi_B.
\]

Now, if \( B, C \in \Sigma_A \) lie in the same \( \Aut_C(S) \)-orbit, say \( C = R \cdot B \) for some \( R \in \Aut_C(S) \), and \( B = \psi \cdot A \), then \( C = R\psi \cdot A \). The cocycle associated to \( C \) is:
\[
    \Phi_C(\sigma) = (R\psi)^{-1} \sigma(R\psi) = \psi^{-1} R^{-1} \sigma R \cdot \sigma \psi = \psi^{-1} \sigma \psi = \Phi_B(\sigma),
\]
since \( R \in \Aut_C(S) \), it is stable under the action of \( \sigma \in G \). This implies \( \Phi \) descends to a map
\[
    \Aut_C(S) \backslash \Sigma_A \to H^1(G, M_A).
\]

To prove injectivity, suppose \( B = \psi_1 \cdot A \), \( C = \psi_2 \cdot A \), and \( \Phi_B = \Phi_C \). Then by definition of cohomologous cocycles, there exists \( N \in M_A \) such that:
\[
    \psi_1^{-1} \sigma \psi_1 = N^{-1} \psi_2^{-1} \sigma \psi_2 \sigma(N) \quad \forall \sigma \in G.
\]
Rewriting, we get:
\[
    \psi_2 N \psi_1^{-1} = \sigma(\psi_2 N \psi_1^{-1}) \quad \forall \sigma \in G,
\]
so \( \psi_2 N \psi_1^{-1} \in \Aut_C(S) \), and
\[
    C = \psi_2 \cdot A = (\psi_2 N \psi_1^{-1}) \cdot B,
\]
which shows \( C \) and \( B \) lie in the same \( \Aut_C(S) \)-orbit. This proves injectivity of the induced map
\[
    \Aut_C(S) \backslash \Sigma_A \xhookrightarrow{\quad\quad}H^1(G, M_A).
\]
\end{proof}

\begin{proposition}\label{finite cohomology of stab}
    For any \( A \in \Omega_{n}(S/C) \), the group \( H^1(G, M_A) \) is finite, and its order is bounded by a constant that  depends only on \( n \), \( g \), \( G \), and \( C \) (not on $A$ nor the family $S$ over $C$).
\end{proposition}

\begin{proof}
Let \( E \to C \) be the finite Galois cover obtained by applying Proposition \ref{prop that div becomes sections} to the curve \( C \). We may further enlarge \( E \) so that the function field of \( E \) contains the Hilbert class field of \( K := \F_q(C) \). Let \( E' := E \times_C C' \), so that \( E' \to C \) is also a finite Galois cover.

Let \( G_1 := \Gal(E'/C) \), \( G_2 := \Gal(E'/C') \). Then \( G_2 \trianglelefteq G_1 \) and \( G_1/G_2 \cong G := \Gal(C'/C) \). By Theorem \ref{inflation restriction}, we obtain an injective map:
\[
    H^1(G, M_A) \xhookrightarrow{\quad\quad} H^1(G_1, M'_A),
\]
where $M'_A$ is the stabilizer of $A$ inside $S^{''} := S\times_C E'$. If we show that $H^1(G_1, M'_A)$ is finite, then $ H^1(G,M_A) $ is finite as well. Therefore, for the rest of the proof we may assume $\rho(A)\in \Omega_{n,C'}(S'/C')$ and the function field of $C'$ contains the Hilbert class field of $C$. 

Let \( S_0 \) denote the generic fiber of \(  S' \to C' \), which is a smooth projective curve over \( K \). Let \( g \) be its genus and \( \chi(S_0) = 2 - 2g \) its Euler characteristic.

The stabilizer \( M_A \subset \Aut_{C'}(S') \) consists of automorphisms fixing  \( \rho(A) \). Any such automorphism restricts to an automorphism of the complement \( S_0 \setminus  \rho(A) \). Note that $\chi(S_0) -n$ is the Euler characteristic of $S_0\backslash\rho(A)$.
We analyze three cases:
\begin{itemize}
    \item \textbf{\( \chi(S_0) - n < 0 \):} The complement \( S_0 \setminus \rho(A) \) has negative Euler characteristic, so its automorphism group is finite by Proposition \ref{negative char implies finite aut} and depends only on \( g \) and \( n \). Then \( M_A \) is finite, and so \( H^1(G, M_A) \) is bounded by  \( |M_A|^{|G|} \).

    \item \textbf{\( \chi(S_0) - n = 0 \):} Then \( \chi(S_0) = n = 2 \), so \( S_0 \cong \mathbb{P}^1 \), and \( A \) is a degree 2 divisor. We can choose coordinates so that \( \rho(A) = \{ [0:1], [1:0] \} \). Then \( M_A \cong (\G_m \rtimes \mathbb{Z}/2\mathbb{Z})(R') \), where \( R' \) is the coordinate ring of \( C' \). There is a short exact sequence:
    \[
        1 \to \G_m(R') \to (\G_m \rtimes \mathbb{Z}/2\mathbb{Z})(R') \to \mathbb{Z}/2\mathbb{Z} \to 1.
    \]
    The group  \( \G_m(R') \) is a finitely generated abelian group and \( \mathbb{Z}/2\mathbb{Z} \) is a finite group. Hence \( H^1(G,\G_m(R')) \), as well as all its twisted forms, are finite by \cite[2, Cor $1.32$]{milne2011class}. Then, by Theorem \ref{coh theo}, the cohomology group \( H^1(G, M_A) \) is finite as well.
    
    \item \textbf{\( \chi(S_0) - n = 1 \):} Then \( \chi(S_0) = 2 \), which implies $g=0$ and \( n = 1 \). This is the case that was excluded at the beginning of this section.
\end{itemize}
\end{proof}

For any $A \in \Omega_n(S/C)$ and any positive integer $m$, we have
\[
F^m_{S/C} \cdot A \in \Omega_n(S^{(q^m)}/C).
\]
As a result, we can apply Proposition \ref{finite cohomology of stab} to all $S^{(q^m)} \to C$ and obtain a uniform bound on the corresponding cohomology groups. This is useful for the following result.

\begin{proposition}\label{finite fiber under finite map}
    Assume $C' \to C$ is a finite Galois cover of curves, and let $\sim'$ denote the equivalence relation on $\Omega_n(S'/C')$. Then any equivalence class in $\Omega_n(S'/C')/\sim'$ intersects only finitely many equivalence classes in $\Omega_n(S/C)/\sim$. More precisely,
    \[
    \#\left\{ \mathcal{B} \in \Omega_n(S/C)/\sim \mid \rho(\mathcal{B}) \cap \mathcal{A} \neq \emptyset \right\} < \infty \quad \text{for all } \mathcal{A} \in \Omega_n(S'/C')/\sim'.
    \]
\end{proposition}

\begin{proof}
Let $G$ be the Galois group of $C' \to C$, and assume that the field of definition of $C'$ is $\mathbb{F}_{q_0}$ where $[\F_{q_0}:\F_q]=k$.
If the family $\pi: S \to C$ is non-isotrivial, then for any $A \in \Omega_n(S/C)$, the fiber of $\rho$ containing the class of $A$ injects into the cohomology group $H^1(G, M_A)$, which is finite. Hence, only finitely many equivalence classes in $\Omega_n(S/C)/\sim$ can map to a fixed class in $\Omega_n(S'/C')/\sim'$.

Now assume that $\pi: S \to C$ is isotrivial, and fix a class $\mathcal{A} \in \Omega_n(S'/C')/\sim'$. Let $d$ and $N_0$ be the constants from Theorem \ref{preperiodic} applied to the family $\pi: S \to C$. Let $B_0, \dots, B_N \in \rho^{-1}(\mathcal{A})$ be elements that are not $\sim$-equivalent in $\Omega_n(S/C)$.

Since all $\rho(B_i)$ lie in the same equivalence class $\mathcal{A}$, for each $i = 1, \dots, N$, there exist integers $r_i, s_i \geq 0$ and an isomorphism
\[
\psi_i: S'^{(q_0^{ r_i})} \xrightarrow{\sim} S'^{(q_0^{ s_i})}
\]
such that
\[
\psi_i \circ F^{r_i}_{S'/C'}(\rho(B_0)) = F^{s_i}_{S'/C'}(\rho(B_i)).
\]

Let $r := \max\{r_1, \dots, r_N\}$. Then for each $i$, there is an isomorphism
\[
\phi_i: S'^{(q_0^{ r})} \xrightarrow{\sim} S'^{(q_0^{ t_i})}, 
\]
where $ t_i := s_i + r - r_i$, satisfying
\[
\phi_i \circ F^r_{S'/C'}(\rho(B_0)) = F^{t_i}_{S'/C'}(\rho(B_i)).
\]

By the pigeonhole principle, among the values $t_i$, there exists a subset $I \subset \{1, \dots, N\}$ of size at least $\lfloor (N - N_0)/d \rfloor$ such that all $t_i$ for $i \in I$ are congruent modulo $d$ and satisfy $t_i > N_0$.

Without loss of generality, suppose $I = \{1, \dots, \lfloor (N - N_0)/d \rfloor \}$. Since $F_{S'/C'} = F_{S/C}^k \times \mathbf{1}_C$, we have for any $i, j \in I$:
\[
\phi_j \circ \phi_i^{-1} \circ \rho\left(F_{S/C}^{k t_i}(B_i)\right) = \rho\left(F_{S/C}^{k t_j}(B_j)\right).
\]
This shows that all the points $F_{S/C}^{k t_i}(B_i)$ for $i \in I$ lie in the same equivalence class under $\sim'$ after applying $\rho$, but they are pairwise non-equivalent under $\sim$. Therefore, they must map to a finite set of representatives inside $\Omega_n(S^{(q^{k t_1})}/C)$ that are pairwise equivalent under $\Aut_{C'}(S'^{(q_0^{ t_1})})$.

Hence, the number $\lfloor (N - N_0)/d \rfloor$ is bounded by the universal constant given in Proposition \ref{finite cohomology of stab}. This implies that $N$ is bounded, and thus only finitely many equivalence classes in $\Omega_n(S/C)/\sim$ can intersect $\mathcal{A}$.
\end{proof}

\begin{proposition}
    Theorem \ref{section theo} implies Theorem \ref{etale theo}.
\end{proposition}

\begin{proof}
Let \( S \to C \) be a smooth and proper family of curves over \( C \), as described in the introduction. By Proposition~\ref{prop that div becomes sections}, there exists a finite, generically Galois cover of curves \( E \to C \) such that there is a natural injection
\[
    \rho \colon \Omega_n(S/C) \xhookrightarrow{\quad} \Omega_{n,E}(S_E/E),
\]
where \( S_E := S \times_C E \). Let \( \sim' \) denote the equivalence relation on \( \Omega_{n,E}(S_E/E) \). Then, by Theorem~\ref{section theo},
\[
    \#\left( \Omega_{n,E}(S_E/E)/\sim' \right) < \infty.
\]
By Proposition~\ref{finite fiber under finite map}, for any \( \mathcal{A} \in \Omega_{n,E}(S_E/E) \), we have
\[
    \#\left\{ \mathcal{B} \in \Omega_n(S/C)/\sim \;\middle|\; \rho(\mathcal{B}) \cap \mathcal{A} \neq \emptyset \right\} < \infty.
\]
Therefore, the number of equivalence classes in \( \Omega_n(S/C)/\sim \) is finite as well.
\end{proof}

\section{Proof of Theorem \ref{section theo} }
\label{final}

By Proposition \ref{finite fiber under finite map}, it suffices to prove Theorem \ref{section theo} when $\pi\colon S \to C$ is either trivial or non-isotrivial. Moreover, we may assume that the Picard group of the base is trivial.

\subsection{Non-isotrivial case}

If $g = 0$, then after a finite base change, the family $\pi\colon S \to C$ becomes isotrivial. Therefore, if $\pi$ is non-isotrivial, the genus of the generic fiber satisfies $g \ge 1$.

Let $S_0$ denote the generic fiber of $\pi\colon S \to C$, which is then a smooth, projective curve of genus $g \ge 1$ over the function field $K := \F_q(C)$ and denote $R:=\F_q[C]$ with  the assumption that $R$ has a trivial Class group. 

\subsection*{$\boldsymbol{g = 1}$}

If $S(C) \neq \emptyset$, then the generic fiber $S_0$ has a $K$-rational point and is thus isomorphic to an elliptic curve over $K$. By \cite{ulmer2011park}, $S_0$ can be given by a Weierstrass equation
\begin{equation}\label{elliptic eq}
Y^2Z + a_1XYZ + a_3YZ^2 = X^3 + a_2X^2Z + a_4XZ^2 + a_6Z^3
\end{equation}
inside $\mathbb{P}^2_K$, where $a_i \in K$. The model $S$ over $C$ corresponds to equation \eqref{elliptic eq} with $a_i \in R$, and discriminant being a unit in $R$, ensuring good reduction over $\mathrm{Spec}\,R$.

Since $R$ is a principal ideal domain, every $K$-rational point of $S_0$ can be represented as $[\alpha : \beta : \gamma]$ with $\alpha, \beta, \gamma \in R$ and $\gcd(\alpha, \beta, \gamma) = 1$, i.e., the ideal $(\alpha, \beta, \gamma)$ is equal to $R$. Now, fix the origin $O := [0 : 1 : 0]$ on $S_0$.

\begin{lemma}
The set of points in \( S_0(K) \) whose reduction modulo every prime ideal \( \pp \) of \( R \) is different from \( O \) is finite.
\end{lemma}

\begin{proof}
    Let $f(X, Y, Z) := Y/Z$. Then $f$ is a nonconstant rational function on $S_0$. By Theorem \ref{siegel like}, the set
\[
W:=\left\{ [\alpha : \beta : \gamma] \in S_0(K) \mid v_\pp(\beta/\gamma) \ge 0 \text{ for all prime ideals } \pp \subset R \right\}
\]
is finite.

 Let $P = [\alpha : \beta : \gamma] \in S_0(K)$ be a point such that for every prime ideal $\pp$ of $R$, the reduction of $P$ modulo $\pp$ is not equal to $O$. Suppose that $\gamma \in \pp$ for some prime ideal $\pp$  of $R$. Substituting into \eqref{elliptic eq}, we find:
\begin{align*}
\beta^2\gamma + a_1\alpha\beta\gamma + a_3\beta\gamma^2 &= \alpha^3 + a_2\alpha^2\gamma + a_4\alpha\gamma^2 + a_6\gamma^3 \\
\implies \alpha^3 &\in \pp \\
\implies \alpha &\in \pp.
\end{align*}
Since $(\alpha, \beta, \gamma) = R$, we must have $\beta \notin \pp$. Thus, the reduction of $P$ modulo $\pp$ is $[0 : 1 : 0]$, which is a contradiction.  Hence \( \gamma \notin \pp \) for any prime ideal \( \pp \subset R \), so \( \gamma \in R^\times \). It follows that \( v_\pp(\beta/\gamma) \ge 0 \) for all \( \pp \), so \( [\alpha : \beta : \gamma] \in W \). Since \( W \) is finite, the result follows.
\end{proof}

Now consider the group law on $S_0(K)$ with identity element $O$. Any $x \in S_0(K)$ induces a translation automorphism $t_x$ of $S_0$. Theorem \ref{ext aut} implies that $t_x$ extends uniquely to an automorphism of $S$ over $C$.

It follows that for $n = 1$, the automorphism group of $S$ over $C$ acts transitively on $\Omega_{1,C}(S/C)$, so there is only one equivalence class. For $n \ge 2$, any $n$-tuple $\{x_1, \dots, x_n\} \in \Omega_{n,C}(S/C)$ is equivalent to $\{O, x_2 - x_1, \dots, x_n - x_1\}$. Since $x_i - x_1 \in S_0(K)$ and the set of such points is finite by the argument above, there are only finitely many such $n$-tuples up to equivalence. This completes the proof in this case. 

\subsection*{$\boldsymbol{g \ge 2}$}

In this case, by Theorem \ref{samuel theo}, the set $S_0(K)$ is finite. Hence, $\Omega_{n,C}(S/C)$, which consists of subsets of $S_0(K)$ of size $n$, is finite.

\subsection{Trivial case}

By applying Proposition \ref{finite fiber under finite map} repeatedly, it is enough to prove the following proposition.

\begin{proposition}\label{final prop}
    Let $S = D \times_{\F_q} C$, where $C$ is a smooth affine curve over $\F_q$ with trivial Picard group, and $D$ is a smooth proper curve over $\F_q$ with at least one $\F_q$-rational point. Then Theorem \ref{section theo} holds for the family $S \to C$.
\end{proposition}

A section of $S \to C$ corresponds bijectively to a morphism $C \to D$. Furthermore, the action of the relative Frobenius $F_{S/C}$ on sections coincides with the action of the absolute Frobenius $F_D$ on $\Hom(C, D)$ via post-composition. Moreover, for any $r \in \mathbb{Z}_{\geq 0}$ and any subset $A \subset S(C)$,
\[
F_{S/C}^r \cdot A \in \Omega_{n,C}(S/C) \iff A \in \Omega_{n,C}(S/C).
\]
We prove the proposition by considering the genus of the curve \(D\). For convenience, let 
\[
m = \#\Bigl(\overline{C}(\overline{\F_q}) \setminus C(\overline{\F_q})\Bigr).
\]

\subsection*{\(\boldsymbol{g \geq 2}\)}

Define 
\[
W = \{ t \colon C \to D \mid t \text{ is separable} \}.
\]
Every \(t \in W\) uniquely extends to a separable morphism \(t \colon \overline{C} \to D\). By applying the Riemann–Hurwitz formula, we obtain
\[
2g_C - 2 = d(2g - 2) + \deg(R),
\]
where \(d\) is the degree of \(t\) and \(R\) is the ramification divisor. This formula shows that \(d\) is bounded. Hence, by Lemma \ref{f. m. sep map}, the number of separable morphisms from \(C\) to \(D\) is finite; that is, \(W\) is finite.

Moreover, any \(A \in \Omega_{n,C}(S/C)\) corresponds to a subset \(A' \subset \Hom(C,D)\) that can be written in the form
\[
A' = \{ x_1, \dots, x_\alpha,\; F_D^{r_1} \circ t_1, \dots, F_D^{r_\beta} \circ t_\beta \},
\]
where
\begin{align*}
    & \alpha + \beta =n\\
    & x_i \in D(\F_q) \quad\, \forall i=1 ,\dots , \alpha \\
    & t_j \in W \quad\quad\quad \forall j=1 ,\dots , \beta\\
    &0\le r_1 \le \dots \le r_\beta
\end{align*}
Since the action of \(F_D\) on constant maps is trivial, the set \(A'\) is equivalent to
\[
B' = \{ x_1, \dots, x_\alpha,\; t_1,\, F_D^{r_2 - r_1} \circ t_2,\, \dots,\, F_D^{r_\beta - r_1} \circ t_\beta \}.
\]
Because both \(D(\F_q)\) and \(W\) are finite sets, there are only finitely many choices for the \(x_i\) and the \(t_j\).

Furthermore, by the definition of \(\Omega_{n,C}(S/C)\), any two distinct elements \(l_1, l_2 \in B'\) must differ on  \(C\). Then, by Proposition \ref{frob control}, each difference \(r_j - r_1\) can only take finitely many possible values. Consequently, the set \(B'\) itself is finite. 

This shows that there are only finitely many equivalence classes for divisors in \(\Omega_{n,C}(S/C)\), completing the proof of Proposition \ref{final prop} in this case.

\subsection*{\(\boldsymbol{g = 1}\)}

Fix a point \(O \in D(\F_q)\), and let \(K\) denote the function field of \(C\). Then \(D' := D \times_{\F_q} K\) is an elliptic curve over \(K\), with identity element \(O\). This curve is the generic fiber of the surface \(S \to C\).

Translation by an element of \(D'(K)\) defines an automorphism of the elliptic curve \(D'\) over \(K\). By Theorem \ref{ext aut}, such a translation extends to a \(C\)-automorphism of \(S = D \times C\). Therefore, for any \(A \in \Omega_{n,C}(S/C)\), there exists an automorphism \(\psi \in \Aut_C(S)\) such that
\[
O \times C \in \psi \cdot A.
\]

Now define
\[
W = \left\{ t \colon C \to D \;\middle|\; t^{-1}(O) \subseteq \overline{C} \setminus C,\; t \text{ is separable} \right\}.
\]
Each \(t \in W\) extends to a finite morphism \(\overline{C} \setminus t^{-1}(O) \to D \setminus \{O\}\). Let \(R\) denote the ramification divisor of this map, and let \(d = \deg(t)\). Then, the Riemann–Hurwitz formula yields:
\begin{align*}
    \chi\bigl(\overline{C} \setminus t^{-1}(O)\bigr) &= d \cdot \chi(D \setminus \{O\}) - \deg(R), \\
    2 - 2g_C - m \le \chi\bigl(\overline{C} \setminus t^{-1}(O)\bigr) &\le d \cdot \chi(D \setminus \{O\}) = d(2 - 2 \cdot 1 - 1), \\
    \implies \quad d &\le 2g_C + m - 2.
\end{align*}

This provides a bound on the degree \(d\), and by Lemma \ref{f. m. sep map}, the set \(W\) is finite. Consequently, any element of \(\Omega_{n,C}(S/C)\) corresponds to a  set of the form
\[
\{ O, x_1, \dots, x_\alpha,\; F_D^{r_1} \circ t_1, \dots, F_D^{r_\beta} \circ t_\beta \},
\]
where
\begin{align*}
    & 1+\alpha + \beta =n\\
    & x_i \in D(\F_q) \quad\quad\, \forall i=1 ,\dots , \alpha \\
    & t_j \in W \quad\quad\quad\quad \forall j=1 ,\dots , \beta\\
    &0\le r_1 \le \dots \le r_\beta
\end{align*}

As in the previous case, this set is equivalent to
\[
\{ O, x_1, \dots, x_\alpha,\; t_1, F_D^{r_2 - r_1} \circ t_2, \dots, F_D^{r_\beta - r_1} \circ t_\beta \},
\]
in which each \(x_i\), \(t_j\), and difference \(r_j - r_1\) has only finitely many possibilities. This completes the proof of Proposition \ref{final prop} in the case \(g = 1\).

\subsection*{\(\boldsymbol{g = 0}\)}

Let \(R := \F_q[C]\), so that \(C = \Spec\, R\), and assume \(D \cong \mathbb{P}^1_{\F_q}\). Then \(S = \mathbb{P}^1_R\), and its group of \(C\)-automorphisms is \(\PGL_2(R)\).

\begin{lemma}
For \(n = 1, 2, 3\), the group \(\PGL_2(R)\) acts transitively on \(\Omega_{n,C}(S/C)\).
\end{lemma}

\begin{proof}
Since \(R\) has a trivial class group, every section of \(\mathbb{P}^1_R\) is of the form \([a:b]\) for some \(a, b \in R\) with \((a,b) = R\). Moreover, two such sections are the same up to units in \(R\).

Let \(A_1 \in \Omega_{1,C}(S/C)\), so \(A_1 = \{[a_1 : b_1]\}\) for some \(a_1, b_1 \in R\) with \((a_1, b_1) = R\). Then there exist \(x, y \in R\) such that \(a_1 x + b_1 y = 1\), and the matrix
\[
\psi_1 = \begin{bmatrix}
b_1 & -a_1 \\
x & y
\end{bmatrix} \in \PGL_2(R)
\]
defines an automorphism that maps \([a_1 : b_1]\) to \([0 : 1]\).

Now let \(A_2 \in \Omega_{2,C}(S/C)\). Without loss of generality, we may assume \([0 : 1] \in A_2\), and denote the other point by \([a_2 : b_2]\). Since \([0 : 1]\) and \([a_2 : b_2]\) are disjoint sections, \(a_2\) must be a unit in \(R\). Then the matrix
\[
\psi_2 = \begin{bmatrix}
a_2^{-1} & 0 \\
-a_2^{-1} b_2 & 1
\end{bmatrix} \in \PGL_2(R)
\]
defines an automorphism that fixes \([0 : 1]\) and maps \([a_2 : b_2]\) to \([1 : 0]\).

Now let \(A_3 \in \Omega_{3,C}(S/C)\). By a similar argument, we may assume \([0 : 1], [1 : 0] \in A_3\), and let the third point be \([a_3 : b_3]\). Since all three sections are disjoint, both \(a_3\) and \(b_3\) must be units in \(R\). Then the matrix
\[
\psi_3 = \begin{bmatrix}
a_3^{-1} & 0 \\
0 & b_3^{-1}
\end{bmatrix} \in \PGL_2(R)
\]
defines an automorphism that fixes \([0 : 1]\) and \([1 : 0]\), and maps \([a_3 : b_3]\) to \([1 : 1]\). This completes the proof.
\end{proof}

Now assume \(n \ge 4\). By a sequence of automorphisms, any element of \(\Omega_{n,C}(S/C)\) is equivalent to
\[
A = \{ [0:1], [1:0], [1:1], [a_4:b_4], \dots, [a_n:b_n] \}.
\]
Let \(V = \{ [0:1], [1:0], [1:1] \}\), and define
\[
W = \left\{ t \colon C \to D \;\middle|\; t(C) \subset D \setminus V,\; t \text{ is separable} \right\}.
\]
Then each \(t \in W\) extends to a finite morphism \(\overline{C} \setminus t^{-1}(V) \to D \setminus V\). Let \(R\) be the ramification divisor of this map, and let \(d = \deg(t)\). Then the Riemann–Hurwitz formula gives:
\begin{align*}
\chi\bigl(\overline{C} \setminus t^{-1}(V)\bigr) &= d \cdot \chi(D \setminus V) - \deg(R), \\
2 - 2g_C - m \le \chi\bigl(\overline{C} \setminus t^{-1}(V)\bigr) &\le d \cdot \chi(D \setminus V) = d(2 - 2 \cdot 0 - 3), \\
\implies \quad d &\le 2g_C + m - 2.
\end{align*}

As before, this bound on \(d\) implies that the set \(W\) is finite by Lemma \ref{f. m. sep map}. Hence, any \(A \in \Omega_{n,C}(S/C)\) corresponds to
\[
\{ [0:1], [1:0], [1:1], x_1, \dots, x_\alpha,\; F_D^{r_1} \circ t_1, \dots, F_D^{r_\beta} \circ t_\beta \},
\]
where:
\begin{align*}
    &3+ \alpha + \beta =n\\
    & x_i \in D(\F_q) \quad\quad\, \forall i=1 ,\dots , \alpha \\
    & t_j \in W \quad\quad\quad\quad \forall j=1 ,\dots , \beta\\
    &0\le r_1 \le \dots \le r_\beta
\end{align*}

As in the previous cases, this is equivalent to
\[
\{ [0:1], [1:0], [1:1], x_1, \dots, x_\alpha,\; t_1, F_D^{r_2 - r_1} \circ t_2, \dots, F_D^{r_\beta - r_1} \circ t_\beta \},
\]
where each \(x_i\), \(t_j\), and \(r_j - r_1\) takes only finitely many values. This completes the proof of Proposition \ref{final prop} in the case \(g = 0\).

\section*{Acknowledgments}

I am deeply grateful to my advisor, Jacob Tsimerman, for his invaluable guidance, insight, and support throughout the development of this work. His thoughtful feedback was essential at every stage of the project. I also sincerely thank Daniel Litt for many helpful conversations, insightful suggestions, and constant encouragement.

 \bibliographystyle{alpha}
 \bibliography{Main}
\end{document}